\documentclass[11pt]{article}
\usepackage{amsmath,amsthm,amsfonts,latexsym,amsopn,verbatim,amscd,amssymb}
\usepackage{adjustbox}
\usepackage{lipsum}
\usepackage{cleveref}
\usepackage{enumerate}
\usepackage[left=2cm, right=2cm, top=2cm]{geometry}

\newtheorem{theorem}{Theorem}[section]

\newtheorem{Example}{Example}[section]

\newtheorem{lemma}{Lemma}[section]

\usepackage{tabularx}
\usepackage{amsmath,amsfonts,amssymb}
\usepackage{graphics, graphicx}
\usepackage{enumitem}
\usepackage{subcaption}
\usepackage{multirow}
\usepackage{authblk}
\usepackage{enumitem,amssymb}
\usepackage{graphicx,float}
\newlist{todolist}{itemize}{2}
\setlist[todolist]{label=$\square$}
\begin{document}
	\title{A parameter uniform method for two-parameter singularly perturbed  boundary value problems with discontinuous data}
	\date{}
\author{Nirmali Roy$^{1*}$, Anuradha Jha$^{2}$}
\date{%
    $^{1,2}$Indian Institute of Information Technology Guwahati, Bongora, India,781015\\
   *Corresponding author(s). E-mails: {nirmali@iiitg.ac.in}
}
	\maketitle
	%%%%%%%%%%%%%%%%%%%%%%%%%%%%%%%%%%%%%%%%%%%%%%%%%%%%%%
	\section*{Abstract}A two-parameter singularly perturbed problem with discontinuous source and convection coefficient is considered in one dimension. Both convection coefficient and source term are discontinuous at a point in the domain. The presence of perturbation parameters results in boundary layers at the boundaries. Also, an interior layer occurs due to the discontinuity of data at an interior point. An upwind scheme on an appropriately defined Shishkin-Bakhvalov mesh is used to resolve the boundary layers and interior layers. A three-point formula is used at the point of discontinuity. The proposed method has first-order parameter uniform convergence. Theoretical error estimates derived are verified using the numerical method on some test problems. Numerical results authenticate the claims made. The use of the Shiskin-Bakhvalov mesh helps achieve the first-order convergence, unlike the Shishkin mesh, where the order of convergence deteriorates due to a logarithmic term. 
		
		\textbf{Keywords}: singular perturbation, Interior layers, boundary layers, two parameters, Shishkin-Bakhvalov mesh.

%%%%%%%%%%%%%%%%%%%%%%%%%%%%%%%%%%%%%%%%%%%%%%%%%%%%%%%5

\section{Introduction}	
\label{sec1}
		Consider a  singularly perturbed reaction-convection-diffusion problem on a unit interval $\Omega= (0,1)$,  with a discontinuous source term and convection coefficient.  The source term and convection coefficient are discontinuous at a point $d \in \Omega$. 
	\begin{equation}
	\begin{aligned}
	\label{twoparameter}
	&\mathcal{L}(y(x))\equiv \epsilon y''(x)+\mu a(x) y'(x)-b(x)y(x)=f(x), \hspace{.5cm} x\in(\Omega^- \cup\Omega^+),\\
	&y(0)=y_0,~y(1)=y_1,\\
	&a(x)\leq - \alpha_1 <0 ~ \text{ for } ~x \in \Omega^-, ~~a(x)\geq  \alpha_2 >0 ~ \text{ for }  ~ x \in \Omega^+ \\
	&\lvert [a](d)\rvert <C,~~\lvert [f](d)\rvert <C,
	\end{aligned}
	\end{equation}
	where $0<\epsilon,\mu\ll 1, \alpha_1, \alpha_2 \in \mathbb{R}$,  and $\Omega^-=(0,d), \Omega^+=(d,1), \Omega=(0,1)$.
	The coefficient $b(x)\geq \gamma >0 $ is sufficiently smooth in $\bar{\Omega}$ and $a(x), f(x)$ are sufficiently smooth in $(\Omega^- \cup\Omega^+)\cup\{0,1\}$.  Also, $a(x),f(x),$ and their derivatives have a jump discontinuity at $d$. The jump in any function at point $d$ is denoted as $[g](d)=g(d+)-g(d-).$ Also, let $\displaystyle \rho= \min_{x \in \bar{\Omega}\backslash \{d\}} \left\{ \bigg\lvert \frac{b(x)}{a(x)} \bigg\rvert \right\}$ and $\alpha= \lvert \min \{\alpha_1, \alpha_2\}\rvert $. Under the above assumptions, the BVP  $\eqref{twoparameter}$  admits a unique solution $y(x) \in \mathcal{C}^1(\Omega) \cap \mathcal{C}^2 (\Omega^- \cup \Omega^+)$.
	
	 The solution of above Equation \eqref{twoparameter}  has boundary layers at both boundaries due to the presence of small perturbation parameters $\epsilon $ and $\mu $. In addition, it has strong interior layers in the neighborhood of $d$ due to the discontinuity of $a$ and $f$ and the sign of $a$ in the domain. The  ratios  $\displaystyle \frac{\epsilon}{\mu^2}$ and $\displaystyle \frac{\mu^2}{\epsilon}$ are crucial in determining the width of boundary and interior layers. So we will analyze the above problems under the following two cases:
	 \ $\sqrt{\alpha}\mu \le\sqrt{\rho \epsilon} $ and $\sqrt{\alpha}\mu > \sqrt{\rho \epsilon} $.
	
	  The singular perturbation problems with/without interior layers appear in several branches of engineering and sciences, such as flows in chemical reactors \cite{Alh}, equations involving modeling  of semiconductor devices \cite{PHS},  simulation of water pollution problems \cite{REILW}, and simulation of  many fluid flows \cite{Hir, KL}. 
	  
	   The study of two-parameter singularly perturbed problems was initiated by O'Malley {\cite{OM1,OM,OM2}}, who examined the asymptotic solution. He noted that the perturbation parameters $\epsilon$ and $\mu$ and their ratio affect the solution of these problems. Much work is done for a singularly perturbed two-parameter reaction-convection-diffusion equation with smooth data \cite{RU,GRP,RPS,Z,ZW}. The study of numerical methods for singularly perturbed two-parameter problems with discontinuity in data is an open area of research with much to explore
	  
	    Some work on singularly perturbed one-parameter problems with a discontinuity is discussed in \cite{Cen,FMRS,FH,FHMRS, Lins02}.

	   Shanti et al. in \cite{SRN06} presented an almost first-order numerical technique for two-parameter singularly perturbed problem with a discontinuous
	   source term. The method comprised of upwind difference scheme on an appropriately defined Shishkin mesh. This result was improved by Prabha et al. in \cite{PCS}, who proposed an almost second-order method on Shishkin mesh comprising the central, mid-point, and upwind difference scheme. A five-point difference scheme was used at the point of discontinuity.  
	   An almost second-order method was given by Chandru et al. in \cite{CPS} for a singularly perturbed two-parameter problem with a discontinuous source term. The method consisted of proper use of upwind, central, and mid-point upwind difference methods on a suitably chosen Shishkin mesh. A three-point scheme was used at the point of discontinuity.
	     
	     Prabha et al. considered the same problem (\ref{twoparameter}) in \cite{TP} and gave an almost first-order method comprising of upwind difference method on a layer adapted Shishkin mesh with a three-point difference scheme at the point of discontinuity.
	     
	     Lin$\ss$ proposed Shishkin-Bakhvalov mesh for a one-parameter singular perturbation problem in \cite{Linss99}. In this mesh, the layer part has graded mesh formed by inverting the boundary layer term. The outer region has a uniform mesh. The transition point is chosen as in Shishkin mesh.

	  In this article,  for equation \eqref{twoparameter}, we have used upwind scheme on Shishkin-Bakhvalov mesh.  At the point of discontinuity a three-point difference scheme is used. The proposed scheme is first-order parameter uniform convergent. Shishkin-Bakhvalov mesh performs better than Shishkin mesh. In Shishkin mesh, the order of convergence is deteriorated due to a logarithmic factor, unlike here. 
	 
	 Throughout this article, $C$ denotes a generic positive constant independent of perturbation parameters, number of mesh points.
	  
	  Here, the  maximum norm on the domain $\Omega$ is denoted by
	  \[ \|v\|_{\Omega}= \max_{x \in \Omega} \lvert v(x)\rvert. \]
	  
	  The structure of the paper is as follows. In Section 2, apriori bounds on the solution are stated, followed by the decomposition of the solution and some derivative bounds in Section 3. The numerical method is proposed in Section 4. Section 5 presents the error estimates for the proposed method. Some numerical results are included in Section 6, which verify the theoretical claims made. A summary of the main results is in Section 7.
\section{Apriori Bounds}\label{sec2}
In this section, we discuss the existence of a unique solution, the minimum principle, stability bound and the apriori bounds for the solution of Equation \eqref{twoparameter}.
	%%%%%%%%%%%%%%%THEOREM 2.1 %%%%%%%%%%%%%%%%%%
	\begin{theorem}
		The SPPs \eqref{twoparameter} has a solution $y(x)\in C^{0}(\bar{\Omega})\cap C^{1}(\Omega)\cap C^{2}(\Omega^{-}\cup \Omega^{+})$.
	\end{theorem}
	%%%%%%%%%PROOF OF THEOREM 2.1  %%%%%%%%%%%%%%%%
\begin{proof}
	The proof is by construction. Let $u_{1}, u_{2}$ be particular solutions to the differential equations
	$$\epsilon u_{1}''(x)+\mu a_{1}(x) u_{1}'(x)-b(x)u_{1}(x)=f(x),~~x \in \Omega^-,$$ and $$\epsilon u_{2}''(x)+\mu a_{2}(x) u_{2}'(x)-b(x)u_{2}(x)=f(x),~~x \in \Omega^+,$$ respectively.
	 The convection coefficients  $a_{1}, a_{2} \in C^{2}(\Omega)$ have the following properties:
	 $$a_{1}(x)=a(x), ~~x \in \Omega^-, ~~a_{1}<0, x \in \Omega$$
	 $$a_{2}(x)=a(x), ~~x \in \Omega^+, ~~a_{2}>0, x \in \Omega.$$
	 Consider the function
	 $$y(x)= \left\{
	 \begin{array}{ll}
	 \displaystyle u_{1}(x)+(y_0-u_{1}(0))\phi_{1}(x)+A\phi_{2}(x), & \hbox{ $x\in \Omega^{-},$ }\\
	 u_{2}(x)+B\phi_{1}(x)+(y_1-u_{2}(1))\phi_{2}(x), & \hbox{ $x\in \Omega^{+}$,}
	 \end{array}
	 \right.~~~$$
	 where $A,B$ are constants chosen appropriately for  $y\in C^{1}(\Omega)$ and $\phi_{1}(x),\phi_{2}(x)$ are the solutions of the boundary value problems
	 	$$\epsilon \phi_{1}''(x)+\mu a_{1}(x) \phi_{1}'(x)-b(x)\phi_{1}(x)=0,~~x \in \Omega,~~\phi_{1}(0)=1,~~\phi_{2}(0)=0,$$ and $$\epsilon \phi_{2}''(x)+\mu a_{2}(x) \phi_{2}'(x)-b(x)\phi_{2}(x)=0,~~x \in \Omega,~~\phi_{2}(0)=0,~~\phi_{2}(1)=1$$
	 	respectively.
	 	
We observe that the function $y$  satisfies $y(0)=y_0, y(1)=y_1,$ and $\epsilon y''(x)+\mu a(x) y'(x)-b(x)y(x)=f(x),~~x \in \Omega^-\cup \Omega^{+}.$ Also on an open interval $(0,1)$, $0<\phi_{i}<1, ~i=1,2.$ Thus, $\phi_{1}, \phi_{2}$ cannot have an internal maximum or minimum, and hence 
$$\phi_{1}'(x)<0,~~\phi_{2}'(x)>0,~~x\in(0,1).$$

For the existence of constants $A$ and $ B$, we require that 
$${\begin{vmatrix}
	 \phi_{2}(d) &-\phi_{1}(d) \\
 \phi_{2}'(d) &-\phi_{1}'(d)
\end{vmatrix} }\ne 0.$$
In fact,  $\phi_{2}'(d) \phi_{1}(d)-\phi_{1}'(d) \phi_{2}(d)>0.$
\end{proof}
In the next result, we prove the minimum principle for the operator $\mathcal{L}$. 
%%%%%%%%%%% THEOREM 2.2- MINIMUM PRINCIPLE %%%%%%%%%%%%%%%
\begin{theorem}{(Minimum Principle)}
	Suppose that a function $z(x)\in C^{0}(\bar{\Omega})\cap C^{1}(\Omega)\cap C^{2}(\Omega^{-}\cup \Omega^{+})$ satisfies $z(0)\ge 0,~z(1)\ge 0$ , $\mathcal{L}z(x)\le 0, \; \forall x \in \Omega^- \cup \Omega^+$ and $[z'](d) \le 0$ then $z(x)\geq 0 \; \forall \; \;x \in \bar{\Omega}$.
\end{theorem}
\begin{proof}
	See \cite{TP} for proof.
\end{proof}
%%%%%%%%%%%THEOREM 2.3  --- STABILITY BOUNDS %%%%%%%%%%%%%%%%
\begin{theorem}
	\label{stability}
	Let $y(x)$ be a solution of (\ref{twoparameter}) then $$\|y\|_{\bar{\Omega}}\le \max\{\lvert y(0)\rvert ,\lvert y(1)\rvert \}+\frac{1}{\gamma} \|f\|_{\Omega^{-}\cup \Omega^{+}}.$$
\end{theorem}
%%%%%%%%%%PROOF OF THEOREM 2.3%%%%%%%%%%%%%%%%
\begin{proof}
	Let $\psi^{\pm}(x)=M\pm y(x),$\\
	where $M=\max\{\lvert y(0)\rvert ,\lvert y(1)\rvert \}+\frac{1}{\gamma} \|f\|_{\Omega^{-}\cup \Omega^{+}}$ and $b(x)>\gamma>0~~\forall x\in \Omega.$

Now  $\psi^{\pm}(0)$ and $\psi^{\pm}(1)$ are non negative. 
For each $x\in \Omega^{-}\cup \Omega^{+}.$
	$$\mathcal{L}\psi^{\pm}(x)=\epsilon \psi^{''\pm}(x)+\mu a(x) \psi^{'\pm}(x)-b(x)\psi^{\pm}(x)\le0.$$
Since $y\in C^{0}(\bar{\Omega}) \cap  C^{1}(\Omega)$	\\
 $$[\psi^{\pm}](d)=\pm [y](d)=0~\text{and}~[\psi^{'\pm}](d)=\pm [y'](d)=0.$$
 It follows from the minimum principle that $\psi^{\pm}(x)\ge0, \forall x \in \bar{\Omega}$, which implies
 $$\|y\|_{\bar{\Omega}}\le \max\{\lvert y(0)\rvert ,\lvert y(1)\rvert \}+\frac{1}{\gamma} \|f\|_{\Omega^{-}\cup \Omega^{+}}.$$
\end{proof}
%%%%%%%%%%%%%% END OF PROOF%%%%%%%%%%%%%%%%

%%%%%%%%%%%%%%THEOREM 2.4 --- DERIVATIVE BOUNDS %%%%%%%%%%%
\begin{theorem}
	\label{deribound}
	Let $y(x)$ be the solution of the problem (\ref{twoparameter}) where $\lvert y(0)\rvert  \le C, \lvert y(1)\rvert \le C$ then for $k=1,2$ it holds that 
	
	$$\|y^{(k)}\|_{\Omega^{-}\cup \;  \Omega^{+}} \le \frac{C}{(\sqrt{\epsilon})^{k}}\bigg(1+\left(\frac{\mu}{\sqrt{\epsilon}}\right)^{k}\bigg) \max \big\{\|y\|,\|f\|\big\},$$
	and
	$$\|y^{(3)}\|_{\Omega^{-}\cup \; \Omega^{+}} \le \frac{C}{(\sqrt{\epsilon})^{3}}\bigg(1+\bigg(\frac{\mu}{\sqrt{\epsilon}}\bigg)^{3}\bigg) \max \big\{\|y\|,\|f\|, \|f'\|\big\}.$$
\end{theorem}
%%%%%%%%PROOF OF THEOREM 2.4 %%%%%%%%%%%%%%%
\begin{proof} We first prove the result for the domain $\Omega^-$. The proof for $\Omega^+ $ follows  the same argument.

	Given any point $x\in(0,d)$, we can construct a neighbourhood $N_{p}=(p, p+r)$ where $r>0$ is such that $x\in N_{p}$ and $N_{p}\subset(0,d)$.  As $y$ is differentiable in $N_{p}$ then the mean value theorem implies that there exists $q\in N_{p}$ such that
	\begin{align*}
	y'(q)&=\frac{y(p+r)-y(p)}{r}\\\implies \lvert y'(q)\rvert &\le\frac{\lvert y(p+r)\rvert +\lvert y(p)\rvert }{r} \le \frac{\|y\| }{r}
	\end{align*}
	Also, $$y'(x)=y'(q)+\int_{q}^{x}y''(\xi)d\xi.$$
	Therefore, from the differential equation \eqref{twoparameter} and using integration by parts, we obtain
	\begin{eqnarray*}
	y(x)&=&y'(q) + \epsilon^{-1}\int_{q}^{x}\left(f(\xi)+ b(\xi)y(\xi)-\mu a(\xi)y'(\xi)\right)d\xi \\
	&=&y'(q)+\epsilon^{-1}\int_{q}^{x}\left(f(\xi)+ b(\xi)y(\xi)+\mu a'(\xi)y(\xi)\right)d\xi -\frac{\mu}{\epsilon}[a(x)y(x)-a(q)y(q)].
	\end{eqnarray*}
	Using the fact that $x-q\le r$ and taking modulus on both sides and after some simplifications, we arrive at the following bound
	$$\lvert y'(x)\rvert \le C\bigg(\frac{1}{r}+\frac{r}{\epsilon}+\frac{\mu}{\epsilon}\bigg)\max\{\|y\|,\|f\|\}.$$
	If we choose $r=\sqrt{\epsilon}$ then the right-hand side of the above expression is minimized with respect to $r$ and we obtain the result for $k=1$,
	$$\|y'\|_{\Omega^{-}} \le \frac{C}{(\sqrt{\epsilon})}\bigg(1+\left(\frac{\mu}{\sqrt{\epsilon}}\right)\bigg) \max \big\{\|y\|,\|f\|\big\},~~ x\in \Omega^{-}.$$
 For $k=2$, the differential equation (\ref{twoparameter}) gives,
 \begin{eqnarray*}
  y^{(2)}(x)&=& \frac{1}{\epsilon} [f(x)+b(x)y(x)-\mu a(x) y'(x)] \\
  \lvert y^{(2)}(x)\rvert&\leq & \frac{1}{\epsilon} \left(\|f\|+\|b\|\|y\| \right)+\frac{\mu}{\epsilon} \|a\|\left(\frac{C}{\sqrt{\epsilon}}\left(1+\frac{\mu}{\sqrt{\epsilon}}\right)\right) \max{\{ \|y\|, \|f\|\}}\\
  &\leq& \frac{C}{\epsilon}\left(1+\frac{\mu}{\sqrt{\epsilon}} + \frac{\mu^2}{\epsilon} \right) \max \big\{\|y\|,\|f\|\big\}.
 \end{eqnarray*}
 On simplifying we arrive at 
 \[\|y^{(2)}\|_{\Omega^-} \leq \frac{C}{(\sqrt{\epsilon})^{2}}\bigg(1+\left(\frac{\mu}{\sqrt{\epsilon}}\right)^{2}\bigg) \max \big\{\|y\|,\|f\|\big\} .\]
  To obtain the required bounds for  $k=3$, we  differentiate the Eq. (\ref{twoparameter}) and arrive at
  \begin{eqnarray*}
   y^{(3)}(x)&=& \frac{1}{\epsilon} [f'(x)+(b(x)y(x)-\mu a(x) y'(x))']. \\
  \end{eqnarray*}
  Taking modulus on both sides and the  bounds for $\|y'\|$ and $\|y''\|$ into consideration, we arrive at,
  \begin{eqnarray*}
  \lvert y^{(3)}(x)\rvert &\leq& \frac{C}{\epsilon \sqrt{\epsilon}}\left( 1+\mu +\sqrt{\epsilon}+\frac{\mu}{\sqrt{\epsilon}}+\frac{\mu^2}{\sqrt{\epsilon}}+\frac{\mu^3}{\epsilon \sqrt{\epsilon}}\right) \max{\{\|y\|,\|f\|,\|f'\|\}}.
  \end{eqnarray*}
  On simplifying, we arrive at
\[\|y^{(3)}\|_{\Omega^{-}} \le \frac{C}{(\sqrt{\epsilon})^{3}}\bigg(1+\bigg(\frac{\mu}{\sqrt{\epsilon}}\bigg)^{3}\bigg) \max \big\{\|y\|,\|f\|, \|f'\|\big\}.\]
\end{proof}
\section{Decomposition of the solution}\label{sec3}
 The bounds presented in the previous section are not sufficient for the error analysis of the discretization methods for the singularly perturbed problems. Thus, to obtain sharp bounds, the solution $y(x)$ is decomposed as in \cite{TP} into layers and regular components as $y(x)=v^{*}(x)+w_{l}^{*}(x)+w_{r}^{*}(x)$. The regular component $v^{*}(x)$ is the solution of 
  \begin{equation}
 \label{regular}
 \left\{
 \begin{array}{ll}
 \displaystyle \mathcal{L}v^{*}(x)=f(x), & x\in \Omega^{-}\cup \Omega^{+},\\
 v^{*}(0)=y(0), & v^{*}(1)=y(1),  \; v^{*}(d-)\; \hbox{ and }  v^{*}(d+) \hbox{ are chosen}.
 \end{array}
 \right.
 \end{equation}
 The singular components $w_{l}^{*}(x)$ and $w_{r}^{*}(x)$ are the solutions of 
\begin{equation} \label{left continuous}
	\begin{aligned} \left\{
 \begin{array}{ll}
 \displaystyle \mathcal{L}w_{l}^{*}(x)=0, & x\in \Omega^{-}\cup \Omega^{+},\\
 w_{l}^{*}(0)=y(0)-v^{*}(0), & w_{l}^{*}(1)=0,\; w_l^*(d-)\hbox{ and }  w_l^{*}(d+) \hbox{ are chosen}. 
 \end{array} \right.
\end{aligned}
\end{equation}
and
\begin{equation} \label{right continuous}
\begin{aligned} \left\{
  \begin{array}{ll} 
 \displaystyle \mathcal{L}w_{r}^{*}(x)=0, & x\in \Omega^{-}\cup \Omega^{+},\\
 w_{r}^{*}(0)=0, & w_{r}^{*}(1)=y(1)-v^{*}(1) 
  \end{array}  \right.
\end{aligned}
\end{equation}
 respectively.
 
 The regular and layer components are further decomposed as
\begin{equation*}
v^{*}(x)=\left\{
 \begin{array}{ll}
 \displaystyle v^{*-}(x), & x\in \Omega^-,\\
 v^{*+}(x), & x\in \Omega^+,
 \end{array}
 \right.\end{equation*}
 $$w_{l}^{*}(x)=\left\{
\begin{array}{ll}
\displaystyle w_{l}^{*-}(x), & x\in \Omega^-,\\
w_{l}^{*+}(x), & x\in \Omega^+,
\end{array}
\right.$$
and
 $$w_{r}^{*}(x)=\left\{
 \begin{array}{ll}
 \displaystyle w_{r}^{*-}(x), &  x\in \Omega^-,\\
 w_{r}^{*+}(x), & x\in \Omega^+.
 \end{array}
 \right.$$
As $y \in \mathcal{C}^1(\Omega)$, we have
$[w_{r}^{*}](d)=-[v^{*}](d)-[w_{l}^{*}](d)$ and $[w_{r}^{*'}](d)=-[v^{*'}](d)-[w_{l}^{*'}](d).$\\
 We will find the bounds on these components for case  $\sqrt{\alpha}\mu \le\sqrt{\rho \epsilon}$ first.
 
   Let us decompose the regular part (similar to \cite{TP}) as $v^{*}(x)=v_{0}^{*}(x)+\sqrt{\epsilon}v_{1}^{*}(x)+\epsilon v_{2}^{*}(x),$ where $v_{0}^{*}(x),v_{1}^{*}(x)$ and $v_{2}^{*}(x)$ be the solution of the following problems:
 \begin{eqnarray*}
 -b(x)v_{0}^{*}(x)&=&f(x), ~~~x\in \Omega^-\cup \Omega^+, \\
 -b(x)v_{1}^{*}(x)&=&-\frac{\mu}{\sqrt{\epsilon}}a(x)v_{0}^{*'}(x)-\sqrt{\epsilon}v_{0}^{*''}(x), ~~~x\in \Omega^-\cup \Omega^+ ,\\
 \mathcal{L}v_{2}^{*}(x)&=&-\frac{\mu}{\sqrt{\epsilon}}a(x)v_{1}^{*'}(x)-\sqrt{\epsilon}v_{1}^{*''}(x)=F(x), ~~~x\in \Omega^-\cup \Omega^+,  \\
 && v_{2}^{*}(0)=v_{2}^{*}(1)=0,~v_{2}^{*}(d-),~v_{2}^{*}(d+) \text{ are chosen suitably,}
 \end{eqnarray*}
 respectively.
 
Also, $v_{2}^{*}(x)\in C^{0}(\bar{\Omega})\cap C^{1}(\Omega)\cap C^{2}(\Omega^{-}\cup \Omega^{+})$.
 %%%%%%%%%%%%%%Theorem 3.1%%%%%%%%%%%%%%%%%
\begin{theorem}
	The regular component $v^{*}(x)$ and its derivatives upto order 3 satisfies the following bounds for $\sqrt{\alpha}\mu\le\sqrt{\rho \epsilon}$\\ 
	$$\|v^{*(k)}\|_{\Omega^{-}\cup
		\Omega^{+}} \le C\bigg(1+\frac{1}{(\sqrt{\epsilon})^{k-2}}\bigg), ~~~~k=0,1,2,3.$$
\end{theorem}
%%%%%%%%%%%%Proof of theorem 3.1%%%%%%
\begin{proof}
	To bound the regular component $v^{*}(x)$, we need to bound $v_{0}^{*}(x),v_{1}^{*}(x)$ and $v_{2}^{*}(x)$. With sufficient smoothness on the co-efficient $b(x)$ in $\bar{\Omega}$ and $a(x), f(x)$ in $(\Omega^{-}\cup \Omega^{+})$, we observed that $v_{0}^{*}(x),v_{1}^{*}(x)$  and its derivatives are bounded. To bound $v_{2}^{*}(x)$,  Theorem (\ref{stability}) gives
	$$\|v_{2}^{*}(x)\|_{\Omega^{-}\cup
		\Omega^{+}}\le\frac{1}{\gamma}[\|v_{1}^{*'}\|+\|v_{1}^{*''}\|]\le C.$$
	Now by Theorem (\ref{deribound})
	\begin{align*}
	\|v_{2}^{*(k)}(x)\|_{\Omega^{-}\cup
		\Omega^{+}}&\le\frac{C}{(\sqrt{\epsilon})^k}\bigg(1+\left(\frac{\mu}{\sqrt{\epsilon}}\right)^k\bigg)\max \big\{\|v_{2}^{*}\|,\|F\|\big\},~( \|F\| \leq C(\|v_{1}^{*'}\|+\|v_{1}^{*''}\|))\\
		&\le\bigg(\frac{C}{\sqrt{\epsilon}}\bigg)^{k},\hspace{.5cm} \text{for}  ~~k=1,2.\\
	\text{Also}~\|v_{2}^{*(3)}(x)\|_{\Omega^{-}\cup
		\Omega^{+}} &\le\left(\frac{C}{\sqrt{\epsilon}}\right)^3\bigg(1+\left(\frac{\mu}{\sqrt{\epsilon}}\right)^3\bigg)\max \big\{\|v_{2}^{*}\|,\|F\|, \|F^{(1)}\|\big\}\\
		&\le\bigg(\frac{C}{\sqrt{\epsilon}}\bigg)^{3}.
	\end{align*}
Using the bounds for $v_{0}^*, v_{1}^*, v_{2}^*$  and its derivatives in the expression for $v^{*}(x)$, we have
$$\|v^{*(k)}\|_{\Omega^{-}\cup
	\Omega^{+}} \le C\bigg(1+\frac{1}{(\sqrt{\epsilon})^{k-2}}\bigg), ~~~~k=0,1,2,3.$$
\end{proof}
\begin{theorem}
	\label{boundswl}
    Let $\sqrt{\alpha}\mu\le\sqrt{\rho \epsilon}$. The singular components $w_{l}^{*}(x)$ and $ w_{r}^{*}(x)$  and their derivatives up to order 3 satisfy the following bounds for $k=0,1,2,3$\\
\[\|w_{l}^{*(k)}(x)\|_{\Omega^{-}\cup
		\Omega^{+}} \le \frac{C}{(\sqrt{\epsilon})^{k}}\left\{
	\begin{array}{ll}
	\displaystyle e^{-\theta_{2}x}, &  \; x\in \Omega^{-},\\
	e^{-\theta_{1}(x-d)}, & \; x\in \Omega^{+},
	\end{array}
	\right.\]
	\[\|w_{r}^{*(k)}(x)\|_{\Omega^{-}\cup
		\Omega^{+}} \le \frac{C}{(\sqrt{\epsilon})^{k}}\left\{
	\begin{array}{ll}
	\displaystyle e^{-\theta_{1}(d-x)}, &  \; x\in \Omega^{-},\\
	e^{-\theta_{2}(1-x)}, & x\in \Omega^{+},
	\end{array}
	\right.\]
 where,
	\[ \theta_{1}=\theta_{2} = \displaystyle \frac{\sqrt{\rho \alpha}}{2\sqrt{\epsilon}}.\]

\end{theorem}
%Note: We have proved the above result for $\theta_{1}=\frac{\sqrt{\rho \alpha}}{2\sqrt{\epsilon}}$ and $\theta_{2}=\frac{\sqrt{\rho \alpha}}{2\sqrt{\epsilon}}$ also.
\begin{proof}
	Consider a barrier function $\xi_{\pm}(x)=Ce^{-\theta_{2}x}\pm w_{l}^{*-}(x),~~~x\in\Omega^{-}=(0,d).$
	For a large C, $\xi_{\pm}(0)\ge0$ and $\xi_{\pm}(d)= Ce^{-\theta_2 d} \pm w_l^{*-}(d^-)\ge0$.
	Now 
	\begin{align*}
	\mathcal{L}\xi_{\pm}(x)&=Ce^{-\theta_{2}x}(\epsilon\theta_{2}^{2}-\mu a(x)\theta_{2}-b(x))\\&\le Ce^{-\theta_{2}x}\bigg(\frac{\rho\alpha}{4}-a(x)\frac{\mu}{\sqrt{\epsilon}}\frac{\sqrt{\rho\alpha}}{2}-b(x)\bigg)\\&\le Ce^{-\theta_{2}x}(\rho\lvert a(x)\rvert -b(x))\\&\le0.
	\end{align*}
	Therefore 
	$$\|w_{l}^{*-}\|\le Ce^{-\theta_{2}x},~x\in\Omega^{-}.$$
	Similarly  choose a barrier function $\xi_{\pm}(x)= C e^{-\theta_{1}(x-d)} \pm w_l^{*+}(x), \; x \in \Omega^+$ with large $C$. Now $\xi_{\pm}(d) \ge 0, \xi_{\pm}(1)\ge 0$ with $\mathcal{L}\xi_{\pm}(x) \le 0$ gives 
	\[\|w_{l}^{*+}\|\le C e^{-\theta_{1}(x-d)}, x \in \Omega^+.\] 
	Using Theorem (\ref{deribound}) on $\Omega^-$ and $\Omega^+$, we obtain the following bounds for the derivatives of $w_l^*$ up to order 3,
\[\|w_{l}^{*(k)}\|\le \frac{C}{(\sqrt{\epsilon})^{k}} \left\{ \begin{array}{ll} \displaystyle e^{-\theta_{2}x},& x\in\Omega^{-},\\
\displaystyle e^{-\theta_{1}(x-d)},& x\in\Omega^{+}.\end{array}\right. \]
	Consider a barrier function $\xi_{\pm}(x)=Ce^{-\theta_{1}(d-x)}\pm w_{r}^{*-}(x),~x\in\Omega^{-}=(0,d).$
	For any large C, $\xi_{\pm}(0)\ge0$ and $\xi_{\pm}(d)\ge0$.
	Now 
	\begin{align*}
	\mathcal{L}\xi_{\pm}(x)&=Ce^{-\theta_{1}(d-x)}(\epsilon\theta_{1}^{2}+\mu a(x)\theta_{1}-b(x))\\&\le Ce^{-\theta_{1}(d-x)}\bigg(\frac{\rho\alpha}{4}+a(x)\frac{\mu}{\sqrt{\epsilon}}\frac{\sqrt{\rho\alpha}}{2}-b(x)\bigg)\\&\le Ce^{-\theta_{1}(d-x)}(\frac{\rho \alpha}{4}+\frac{\rho a(x)}{2}-b(x))\\&\le Ce^{-\theta_{1}(d-x)}(-\frac{\rho \alpha}{4}-b(x))\\&\le0.
	\end{align*}
	Therefore 
	$$\|w_{r}^{*-}\|\le Ce^{-\theta_{1}(d-x)},~x\in\Omega^{-}.$$
	For $x \in \Omega^+=(d,1),$ choose the barrier function $\xi_{\pm}(x)=Ce^{-\theta_{2}(1-x)}\pm w_{r}^{*+}(x),~x\in\Omega^{+}$, with large $C$. This gives 
	$\xi_{\pm}(d)\ge0, \xi_{\pm}(1)\ge0$ and $\mathcal{L}\xi_{\pm}(x) \leq 0, x \in \Omega^+$ gives 
	\[\|w_{r}^{*+}\|\le C e^{-\theta_{2}(1-x)}, x \in \Omega^+.\]
	By Theorem (\ref{deribound}), we have the following bounds for the derivatives of $w_r^*$ of order up to 3,
	$$\|w_{r}^{*(k)}\|\le  \frac{C}{(\sqrt{\epsilon})^{k}} \left\{ \begin{array}{ll}
	e^{-\theta_{1}(d-x)},&~x\in\Omega^{-}, \\
 e^{-\theta_{2}(1-x)}, &~x\in\Omega^{+}.
 \end{array}\right. $$
\end{proof}

Consider the  case: $\sqrt{\alpha}\mu>\sqrt{\rho \epsilon}$.

Let $v^*$ be the regular component of the solution $y$ of the Eq. \eqref{twoparameter}. Let us decompose it \cite{TP} as

$v^{*}(x)=v_{0}^{*}(x)+\epsilon v_{1}^{*}(x)+\epsilon^{2}v_{2}^{*}(x),$ where $v_{0}^{*}(x),v_{1}^{*}(x)$ and $v_{2}^{*}(x)$ are the solution of the following problems respectively:
\begin{eqnarray*} 
\mathcal{L}_{\mu} v_0^*(x) &\equiv& \mu a(x)v_{0}^{*'}(x)-b(x)v_{0}^{*}(x)=f(x), ~~x\in \Omega^-\cup \Omega^+,~ v_{0}^{*}(x)=y(0),~v_{0}^{*}(1)=y(1),\\
\mathcal{L}_{\mu} v_1^*(x) &=& -v_{0}^{*''}(x), ~~~x\in \Omega^-\cup \Omega^+,~~v_{1}^{ *}(0)=v_{1}^{*}(1)=0,\\
\mathcal{L}v_{2}^{*}(x)&=&-v_{1}^{*''}(x), ~~~x\in \Omega^-\cup \Omega^+, v_{2}^{*}(0)=v_{2}^{*}(1)=0, 
\end{eqnarray*}
$v_{2}^{*}(d-),~v_{2}^{*}(d+) ~$\text{ are chosen suitably, and} ~$v_{2}^{*}(x)\in C^{0}(\bar{\Omega})\cap C^{1}(\Omega)\cap C^{2}(\Omega^{-}\cup \Omega^{+})$.
%%%%%%%%theorem 3.3%%%%%%%%%%%%%%%%%%%%%%%%%%

The proof of the next theorem follows the argument presented in \cite[Section~ 3]{GRP}  closely.
\begin{theorem}
	Let  $\sqrt{\alpha}\mu>\sqrt{\rho \epsilon}$. The regular component $v^{*}(x)$ and its derivatives up to order 3 satisfies the following bounds\\ 
	$$\|v^{*(k)}\|_{\Omega^{-}\cup
		\Omega^{+}} \le C\bigg(1+\bigg(\frac{\epsilon}{\mu}\bigg)^{(2-k)}\bigg),~~~~k=0,1,2,3.$$
\end{theorem}
\begin{proof} 
	For $x \in \Omega^-, $ the coefficient $a<0$  and $b>0$. Hence, we have that

	\begin{equation}
	\label{firstmin}
	\mathcal{L}_{\mu}z(x)\lvert_{(0,d)}\le0 ~\text{and}~ z(0)\ge0, ~\text{then}~ z(x)\big\rvert _{[0,d)}\ge0.
	\end{equation} 
	Also for  $x \in \Omega^+, $ the coefficients $a>0$  and $b>0$, we have the following result
	\begin{equation}
	\label{firstmin1}
	\mathcal{L}_{\mu}z(x)\lvert_{(d,1)}\le0 ~\text{and}~ z(1)\ge0, ~\text{then}~ z(x)\big\rvert_{(d,1]}\ge0.
	\end{equation}

	We further decompose the component $v^{*}_{0}(x), x\in\Omega^{-}\cup\Omega^{+}$, as follows ,
\[	v^{*}_{0}(x)=s_{0}(x)+\mu s_{1}(x)+\mu^{2}s_{2}(x)+\mu^{3}s_{3}(x),\]
	where	$\displaystyle s_{0}(x)=\frac{-f(x)}{b(x)},~~~s_{1}(x)=\frac{a(x)s_{0}'(x)}{b(x)},~~~s_{2}(x)=\frac{a(x)s_{1}'(x)}{b(x)},$ and
	\begin{equation}
	\label{s3}
	\mathcal{L}_{\mu}s_{3}(x)=-a(x)s_{2}'(x), x \in \Omega^-\cup \Omega^+,~~s_{3}(0)=s_{3}(1)=0.
	\end{equation}
	Assuming sufficient smoothness of the coefficients, the $s_i,i=0,1,2$ and its derivatives are bounded independently of the perturbation parameter $\mu$. In particular, if $ b \in \mathcal{C}^7(\Omega), a,f \in \mathcal{C}^7(\Omega^-\cup \Omega^+)$ we have
	\begin{eqnarray*} \|s_{0}^{(i)}\|&\le& C,~~~~0\le i \le 7,\\
	\|s_{1}^{(i)}\|&\le& C,~~~~0\le i \le 6,\\
\|s_{2}^{(i)}\|&\le& C,~~~~~0\le i \le 5.
\end{eqnarray*}
	Using (\ref{firstmin}) and (\ref{firstmin1}) we deduce that $\|s_{3}\|\le C$ and then from (\ref{s3}) we obtain
$$\|s_{3}^{(i)}\|\le \frac{C}{\mu^{i}},~~~~0\le i \le 5.$$
We use these bounds for $s_{0}(x),s_{1}(x),s_{2}(x)$ and $s_{3}(x)$ to obtain \[
\|v_{0}^{*(i)}\|_{\Omega^-\cup \Omega^+}\le C\bigg(1+\frac{1}{\mu^{i-3}}\bigg), 0\le i \le 5.\]
Now to bound $v_{1}^{*}(x)$ we decompose $v_{1}^{*}(x), x \in\Omega^{-}\cup\Omega^{+}$, as follows\\
$$v^{*}_{1}(x)=\rho_{0}(x)+\mu \rho_{1}(x)+\mu^{2}\rho_{2}(x),$$
where	$ \displaystyle \rho_{0}(x)=\frac{v_{0}^{*''}(x)}{b(x)}, \; \rho_{1}(x)=\frac{a(x)\rho_{0}'(x)}{b(x)},$ and\\
\begin{eqnarray}
\label{rho2}
\mathcal{L}_{\mu}\rho_{2}(x)&=&-a(x)\rho_{1}'(x),\; x \in\Omega^{-}\cup\Omega^{+},\\
\rho_{2}(0)&=&\rho_{2}(1)=0 \nonumber.
\end{eqnarray}
Assuming sufficient smoothness of the coefficients, we have
$$\|\rho_{0}^{(i)}\|_{\Omega^{-}\cup\Omega^{+}}\le C\bigg(1+\frac{1}{\mu^{i-1}}\bigg),~~~~0\le i \le 5$$
and
$$\|\rho_{1}^{(i)}\|_{\Omega^{-}\cup\Omega^{+}}\le \frac{C}{\mu^{i}},~~~~0\le i \le 4$$
Using (\ref{firstmin}), (\ref{firstmin1}) and (\ref{rho2}) we obtain
$$\|\rho_{2}^{(i)}\|_{\Omega^{-}\cup\Omega^{+}}\le \frac{C}{\mu^{i+1}}~~~~0\le i \le 4.$$
We use these bounds for $\rho_{0}(x),\rho_{1}(x)$ and $\rho_{2}(x)$ to obtain 
\[\|v_{1}^{*(i)}\|_{\Omega^{-}\cup\Omega^{+}}\le C(1+\mu^{1-i}), \quad 0\le i \le 3.\]
To bound $v_{2}^{*}(x), \; x\in\Omega^{-}\cup\Omega^{+}$ we use the differential equation satisfied by it.
\begin{equation}
\label{v2}
\mathcal{L}v_{2}^{*}(x)=-v_{1}^{*''}(x),~~v^{*}_{2}(0)=v^{*}_{2}(1)=0.
~~v_{2}^{*}(d-), v_{2}^{*}(d+)~\text{are chosen.}
\end{equation}
Application of Theorem (\ref{stability}) gives
$$\|v_{2}^{*}\|_{\Omega^{-}\cup\Omega^{+}}\le \max\{\lvert v_{2}^{*}(0)\rvert,\lvert v_{2}^{*}(1)\rvert \}+\frac{1}{\gamma}\|v_{1}^{*''}\|\le\frac{C}{\mu^{2}}.$$
By Theorem (\ref{deribound}) we have
$$\|v_{2}^{*(i)}\|_{\Omega^{-}\cup\Omega^{+}}\le \frac{C}{\sqrt{\epsilon}^{(i)}}\bigg(1+\bigg(\frac{\mu}{\sqrt{\epsilon}}\bigg)^{i}\bigg)\frac{1}{\mu^{2}},~~\text{for}~  i=1,2.$$
Differentiating the equation (\ref{v2}) gives
$$\|v_{2}^{*(3)}\|_{\Omega^{-}\cup\Omega^{+}}\le C\frac{\mu}{\epsilon^{3}}.$$
Substituting  these bounds for $v_{0}^{*}(x),v_{1}^{*}(x), v_{2}^{*}(x)$ and their derivatives into the equation for $v^{*}(x)$ gives us
$$\|v^{*(k)}\|_{\Omega^{-}\cup
	\Omega^{+}} \le  C\bigg(1+\bigg(\frac{\epsilon}{\mu}\bigg)^{(2-k)}\bigg),~~~~k=0,1,2,3.$$
\end{proof}
\begin{theorem}
	\label{wr}
	Let $\sqrt{\alpha}\mu>\sqrt{\rho \epsilon}$. The singular components $w_{l}^{*}(x)$ and $w_{r}^{*}(x)$ satisfy the following bounds for $k=0,1,2,3$\\
	$$\|w_{l}^{*(k)}(x)\|_{\Omega^{-}\cup
		\Omega^{+}} \le C\left\{
	\begin{array}{ll}
	\displaystyle \bigg(\frac{1}{\mu}\bigg)^{k}e^{-\theta_{2}x}, & \hbox{ $x\in \Omega^{-}$, }\\
	\displaystyle\bigg(\frac{\mu}{\epsilon}\bigg)^{k}e^{-\theta_{1}(x-d)}, & \hbox{ $x\in \Omega^{+}$,}
	\end{array}
	\right.$$
	$$\|w_{r}^{*(k)}(x)\|_{\Omega^{-}\cup
		\Omega^{+}} \le C\left\{
	\begin{array}{ll}
	\displaystyle \bigg(\frac{\mu}{\epsilon}\bigg)^{k}e^{-\theta_{1}(d-x)}, & \hbox{ $x\in \Omega^{-}$, }\\
	\displaystyle\bigg(\frac{1}{\mu}\bigg)^{k}e^{-\theta_{2}(1-x)}, & \hbox{ $x\in \Omega^{+}$,}
	\end{array}
	\right.$$
	where\\
\[\theta_{1}= \frac{\alpha\mu}{2\epsilon}, \quad \theta_{2}=\frac{\rho}{2\mu}. \]
\end{theorem}
\begin{proof}
In region $\Omega^{-},$ we will find the bound for the left and right layer term. For the left layer, consider a barrier function $\xi_{\pm}(x)=Ce^{-\theta_{2}x}\pm w_{l}^{*-}(x),~~~x\in\Omega^{-}=(0,d).$
	For a large C, $\xi_{\pm}(0)\ge0$ and $\xi_{\pm}(d)\ge0$.
	Now 
	\begin{align*}
	\mathcal{L}\xi_{\pm}(x)&=Ce^{-\theta_{2}x}(\epsilon\theta_{2}^{2}-\mu a(x)\theta_{2}-b(x))\\&\le Ce^{-\theta_{2}x}\bigg(\frac{\rho\alpha}{4}+\lvert a(x)\rvert \frac{\rho}{2}-b(x)\bigg)\\&\le Ce^{-\theta_{2}x}(\rho\lvert a(x)\rvert -b(x))\\&\le0.
	\end{align*}
	therefore 
	$$\|w_{l}^{*-}\|\le Ce^{-\theta_{2}x},~x\in\Omega^{-}.$$
For the right layer term, consider a barrier function $\xi_{\pm}(x)=Ce^{-\theta_{1}(d-x)}\pm w_{r}^{*-}(x),~x\in\Omega^{-}=(0,d).$
	For any large C, $\xi_{\pm}(0)\ge0$ and $\xi_{\pm}(d)\ge0$.
	Now 
	\begin{align*}
	\mathcal{L}\xi_{\pm}(x)&=Ce^{-\theta_{1}(d-x)}(\epsilon\theta_{1}^{2}+\mu a(x)\theta_{1}-b(x))\\&\le
	Ce^{-\theta_{1}(d-x)}\bigg(\frac{\alpha^{2}\mu^{2}}{4\epsilon}+\mu a(x)\frac{\alpha\mu}{2\epsilon}-b(x)\bigg)\\&\le0.
	\end{align*}
	Therefore 
	$$\|w_{r}^{*-}\|\le Ce^{-\theta_{1}(d-x)},~x\in\Omega^{-}.$$
		In a similar way, we can prove the bounds for $w_{l}^{*+}(x)$ and $ w_{r}^{*+}(x)$ in the region $\Omega^{+}$. 
		The bounds for higher derivatives of $w_l^*$ and $w_r^*$ can be proved using the techniques given in \cite{FHMRS1, RPS}.

\end{proof}

The unique solution $y(x)$ of the problem (\ref{twoparameter}) is now given by
$$y(x)=\left\{
\begin{array}{ll}
\displaystyle v^{*-}(x)+w_{l}^{*-}(x)+w_{r}^{*-}(x), & \hbox{ $x\in (0,d),$ }\\
(v^{*-}+w_{l}^{*-}+w_{r}^{*-})(d-)=(v^{*+}+w_{l}^{*+}+w_{r}^{*+})(d+), & \hbox{ $x=d$,}\\
v^{*+}(x)+w_{l}^{*+}(x)+w_{r}^{*+}(x), & \hbox{ $x\in (d,1).$ }
\end{array}
\right.$$
\section{Discrete problem}\label{sec4}
The differential equation \eqref{twoparameter} is discretized using the upwind finite difference method on a suitably constructed Shishkin-Bakhvalov mesh. The domain $\bar{\Omega}=[0,1]$ is subdivided into six subintervals as follows
\[ \bar{\Omega}=[0,\sigma_{1}]\cup[\sigma_{1},d-\sigma_{2}]\cup[d-\sigma_{2},d]\cup[d,d+\sigma_{3}]\cup[d+\sigma_{3},1-\sigma_{4}]\cup[1-\sigma_{4},1].\]

Let $\bar{\Omega}_{N}=\{x_{i}\}_{0}^{N}$ denotes the mesh points with a point of discontinuity at the point $\displaystyle x_{\frac{N}{2}}=d.$ The interior points of the mesh are denoted by $\displaystyle \Omega_{N}=\{x_{i}:1\le i \le \frac{N}{2}-1\}\cup\{x_{i}:\frac{N}{2}+1\le i \le N-1\}.$ Let 
$\displaystyle \Omega_{N}^-= \{x_i, 1\leq i \leq \frac{N}{2}-1\}$ and $\displaystyle \Omega_{N}^+= \{x_i,  \frac{N}{2}+1\leq i \leq N-1\}.$ 
The transition points in $\bar{\Omega}$ are:
\begin{eqnarray*}
\sigma_{1}&=&\frac{4}{\theta_{2}}\ln N,\quad \sigma_{2}=\frac{4}{\theta_{1}}\ln N,\\
\sigma_{3}&=&\frac{4}{\theta_{1}}\ln N,\quad \sigma_{4}=\frac{4}{\theta_{2}}\ln N.
\end{eqnarray*}

On the sub-intervals $[0,\sigma_{1}],[d-\sigma_{2},d],[d,d+\sigma_{3}]$ and  $[1-\sigma_{4},1]$ a graded mesh of $\frac{N}{8}+1$ mesh points is constructed  by inverting the layer function $e^{-\theta_{2}x}, e^{-\theta_{1}(x-d)}, e^{-\theta_{1}(d-x)}$ and $e^{-\theta_{2}(1-x)}$ in the above sub-intervals respectively. On $[\sigma_{1},d-\sigma_{2}]$ and $[d+\sigma_{3},1-\sigma_{4}]$ a uniform mesh of $\frac{N}{4}+1$ mesh points is taken.
We assume that for the case $\sqrt{\alpha}\mu \le \sqrt{\rho\epsilon}$, $\sqrt{\epsilon}<N^{-1}$  and for $\sqrt{\alpha}\mu > \sqrt{\rho\epsilon}, \; \max{\{\epsilon/\mu, \mu\}} < N^{-1}$, otherwise the boundary layers could be resolved by standard uniform mesh.

The mesh points are given by
\begin{equation*}
x_{i}=\left\{
\begin{array}{ll}
\displaystyle -\frac{8}{\theta_{2}}\log\bigg(1+\frac{8i}{N}\bigg(\frac{1}{\sqrt{N}}-1\bigg)\bigg), & 0\le i \le \frac{N}{8}, \vspace{.2cm} \\
\displaystyle\sigma_{1}+\frac{(d-\sigma_{1}-\sigma_{2})\bigg(\frac{i}{N}-\frac{1}{8}\bigg)}{\frac{1}{4}}, & \frac{N}{8}\le i \le \frac{3N}{8}, \vspace{.2cm}\\
\displaystyle d+\frac{8}{\theta_{1}}\log\bigg(\frac{8i}{N}\bigg(1-\frac{1}{\sqrt{N}}\bigg)+\frac{4}{\sqrt{N}}-3\bigg), & \frac{3N}{8}\le i \le \frac{N}{2}, \vspace{.2cm}\\
\displaystyle d-\frac{8}{\theta_{1}}\log\bigg(\frac{8i}{N}\bigg(\frac{1}{\sqrt{N}}-1\bigg)+5-\frac{4}{\sqrt{N}}\bigg), & \frac{N}{2}\le i \le \frac{5N}{8},\vspace{.2cm} \\
\displaystyle d+\sigma_{3}+\frac{(1-d-\sigma_{3}-\sigma_{4})\bigg(\frac{i}{N}-\frac{5}{8}\bigg)}{\frac{1}{4}}, & \frac{5N}{8}\le i \le \frac{7N}{8}, \vspace{.2cm} \\
\displaystyle1+\frac{8}{\theta_{2}}\log\bigg(\frac{8i}{N}\bigg(1-\frac{1}{\sqrt{N}}\bigg)+\frac{8}{\sqrt{N}}-7\bigg), &  \frac{7N}{8}\le i \le N.  
\end{array}
\right.
\end{equation*}

The mesh generating function $\phi$, maps a uniform mesh $\xi$ onto a layer adapted mesh in $x$ by $x=\phi(\xi)$. The mesh in terms of the mesh generating function can be written as:
\[	x_i=\phi(\xi_i)=\left\{
\begin{array}{ll}
\displaystyle \frac{8}{\theta_{2}}\phi_{1}(\xi_i), &  0\le i \le \frac{N}{8},\\
\displaystyle \sigma_{1}+\frac{(d-\sigma_{1}-\sigma_{2})(\xi_i-\frac{1}{8})}{\frac{1}{4}}, & \frac{N}{8}\le i \le \frac{3N}{8},\\
\displaystyle d-\frac{8}{\theta_{1}}\phi_{2}(\xi_i), & \frac{3N}{8}\le i \le \frac{N}{2}, \\
\displaystyle d+\frac{8}{\theta_{1}}\phi_{3}(\xi_i), & \frac{N}{2}\le i \le \frac{5N}{8}, \vspace{.3cm}\\
\displaystyle d+\sigma_{3}+\frac{(1-d-\sigma_{3}-\sigma_{4})(\xi_i-\frac{5}{8})}{\frac{1}{4}}, & \frac{5N}{8}\le i \le \frac{7N}{8},\\
\displaystyle1-\frac{8}{\theta_{2}}\phi_{4}(\xi_i), & \frac{7N}{8}\le i \le N, 
\end{array}
\right.\]
with $\displaystyle \xi_{i}=\frac{i}{N}$. The functions $\phi_{1}, \phi_{3}$ are monotonically increasing on $[0, \frac{1}{8}]$ and $[\frac{1}{2}, \frac{5}{8}]$ respectively. And $\phi_{2}, \phi_{4}$ are monotonically decreasing on $[\frac{3}{8}, \frac{1}{2}]$ and $[\frac{7}{8}, 1]$ respectively. These mesh generating functions $\phi_{i}$'s are defined with the help of corresponding mesh characterizing functions $\psi_{i}$'s as
$$\psi_{i}(\xi)=\exp(-\phi_{i}(\xi)),~~ i=1,2,3,4.$$
\begin{lemma}
	\label{assump}
	We assume that the mesh-generating functions $\phi_{1}, \phi_{2}, \phi_{3}$ and $\phi_{4}$ satisfy the following conditions
$$\max\limits_{\xi\in [0,\frac{1}{8}]}\lvert \phi_{1}^{'}(\xi)\rvert \le CN,~~~~\max\limits_{\xi\in [\frac{3}{8},\frac{1}{2}]}\lvert \phi_{2}^{'}(\xi)\rvert \le CN,$$
$$\max\limits_{\xi\in [\frac{1}{2},\frac{5}{8}]}\lvert \phi_{3}^{'}(\xi)\rvert \le CN,~~~~\max\limits_{\xi\in [\frac{7}{8},1]}\lvert \phi_{4}^{'}(\xi)\rvert \le CN$$  and
$$\int_{0}^{\frac{1}{8}}\{\phi_{1}^{'}(\xi)\}^{2}d\xi\le CN,~~~~~\int_{\frac{3}{8}}^{\frac{1}{2}}\{\phi_{2}^{'}(\xi)\}^{2}d\xi\le CN,$$
$$\int_{\frac{1}{2}}^{\frac{5}{8}}\{\phi_{3}^{'}(\xi)\}^{2} d\xi\le CN,~~~~~ \int_{\frac{7}{8}}^{1}\{\phi_{4}^{'}(\xi)\}^{2}d\xi\le CN.$$
\end{lemma}
\begin{proof}
  The mesh-generating functions $\phi_{1}(\xi)=-\log\bigg[1-8\xi\bigg(\frac{1}{\sqrt{N}}-1\bigg)\bigg],~~\xi\in [0,\frac{1}{8}].$\\
  Therefore, 
  $$\lvert \phi_{1}'(\xi)\rvert\le\frac{8\sqrt{N}}{\sqrt{N}+(1-\sqrt{N})}\le8\sqrt{N}\le CN.$$
  Also mesh characterizing function
 \begin{align*}
     \psi_{1}(\xi)&=\exp(-\phi_{1}(\xi)),~~\xi \in \bigg[0,\frac{1}{8}\bigg]\\&=1+\bigg( \frac{1}{\sqrt{N}}-1\bigg)8\xi\\\psi_{1}'(\xi)&=\bigg( \frac{1}{\sqrt{N}}-1\bigg)8\\\implies\lvert \psi_{1}'(\xi)\rvert &\le 8,~~\xi \in \bigg[0,\frac{1}{8}\bigg].
 \end{align*}
  Similarly, we can prove the bounds for remaining functions in the intervals $[\frac{3}{8},\frac{1}{2}], [\frac{1}{2},\frac{7}{8}]$ and $[\frac{7}{8},1].$
\end{proof}
Using this Lemma (\ref{assump}) we see that for $0\le i \le \frac{N}{8}$, 
\[h_{i}=x_{i}-x_{i-1}=\frac{8}{\theta_{2}}(\phi_{1}(\xi_{i})-\phi_{1}(\xi_{i-1}))\le\frac{8}{\theta_{2}}(\xi_{i}-\xi_{i-1})\max\limits_{\xi\in [0,\frac{1}{8}]}\lvert \phi_{1}^{'}(\xi)\rvert \le \frac{C}{\theta_{2}} .\]

Similarly, we  show that 
\begin{equation*}
h_i \leq  \left\lbrace
\begin{array}{ll} \displaystyle 
\frac{8}{\theta_{1}}(\xi_{i}-\xi_{i-1})\max\limits_{\xi\in [\frac{3}{8},\frac{1}{8}]}\lvert  \phi_{2}^{'}(\xi)\rvert  \le \frac{C}{\theta_{1}}, & \quad \frac{3N}{8}\le i \le \frac{N}{2}  \vspace{.3cm}\\
\displaystyle \frac{8}{\theta_{1}}(\xi_{i}-\xi_{i-1})\max\limits_{\xi\in [\frac{1}{2},\frac{5}{8}]}\lvert \phi_{3}^{'}(\xi)\rvert \le \frac{C}{\theta_{1}}, & \quad  \frac{N}{2}\le i \le \frac{5N}{8} \vspace{.3cm}\\
\displaystyle \frac{8}{\theta_{2}}(\xi_{i}-\xi_{i-1})\max\limits_{\xi\in [\frac{7}{8},1]}\lvert \phi_{4}^{'}(\xi)\rvert \le \frac{C}{\theta_{2}}, &
\quad  \frac{7N}{8}\le i \le N.
\end{array} 
\right.
\end{equation*}

On the Shishkin-Bakhvalov mesh defined above, we use  upwind finite difference method to discretize the differential equation (\ref{twoparameter}).
We define the difference scheme as: Find $Y(x_{i}),~\forall~x_{i}\in \bar{\Omega}_N$ such that:
\begin{equation}
\begin{aligned}
\label{DE}
&\mathcal{L}^{N}Y(x_{i})\equiv \epsilon \delta^{2}Y(x_{i})+\mu a(x_{i})D^{*}Y(x_{i})-b(x_{i})Y(x_{i})=f(x_{i}), ~x_i \in \Omega_N\\
&Y(0)=y(0),~~~Y(1)=y(1),  \\
&D^{-}Y\left(x_{\frac{N}{2}}\right)=D^{+}Y\left(x_{\frac{N}{2}}\right), 
\end{aligned}
\end{equation} 
where 
\[D^{+}Y(x_{i})=\frac{Y(x_{i+1})-Y(x_{i})}{x_{i+1}-x_{i}},~~~~~~ D^{-}Y(x_{i})=\frac{Y(x_{i})-Y(x_{i-1})}{x_{i}-x_{i-1}},\]
\[D^{*}Y(x_{i})=\left\{
\begin{array}{ll}
\displaystyle D^{-}Y(x_{i}), & i<\frac{N}{2}, \vspace{0.2cm}\\
D^{+}Y(x_{i}), &  i>\frac{N}{2},
\end{array}
\right. \quad \delta^{2}Y(x_{i})=\frac{2(D^{+}Y(x_{i})-D^{-}Y(x_{i}))}{x_{i+1}-x_{i-1}}.\]

The following lemma demonstrates that the finite difference operator $\mathcal{L}^N$ has characteristics that are similar to those of the differential operator $\mathcal{L}.$
%%%%%%%%%%%%%%  Lemma 4.1 %%%%%%%%%%%%%%%%%%%%
\\
\begin{lemma}{\bf Discrete minimum principle:}
	\label{minimun}
	Suppose that a mesh function $Y(x_{i})$ satisfies $Y(0)\ge0,~~ Y(1)\ge 0,~~\mathcal{L}^{N} Y(x_{i})\le0,~~ \forall~ x_{i}\in\Omega_{N}$, and  $D^{+}Y(x_{\frac{N}{2}})-D^{-}Y(x_{\frac{N}{2}})\le0$ then $Y(x_{i})\ge0, ~~\forall~x_{i}\in \bar{\Omega}_{N}.$ 
\end{lemma}
\begin{proof} We refer to \cite{TP} for proof.
\end{proof}
%%%%%%%%%%%%%%%Lemma 4.2%%%%%%%%%%%%%%%%%
\begin{lemma}
	If $Y(x_{i}), x_{i}\in\bar{\Omega}_{N}$ is a mesh function  satisfying the difference scheme \eqref{DE}, then $ \|Y\|_{\bar{\Omega}_{N}}\le C$.
\end{lemma}
\begin{proof}
	Define the mesh function for $x_{i}\in\bar{\Omega}_{N}$, as
	\[\omega^{\pm}(x_{i})=M\pm Y(x_{i}),\]
	where $M=\max\{\lvert Y(0)\rvert ,\lvert Y(1)\rvert \}+\frac{1}{\gamma} \|f\|_{\Omega^-\cup \Omega^+}.$
	Now, $\psi^{\pm}(0)$ and $\psi^{\pm}(1)$ are non negative.
	For  $x_{i}\in \Omega_{N},$
	$$\mathcal{L}^N\omega^{\pm}(x_{i})=
	-b(x_{i})M \pm \mathcal{L}^N Y(x_{i})= -b(x_{i})M \pm f(x_i) \le0.$$
	Also $$D^{+}\omega^{\pm}\big(x_{\frac{N}{2}}\big)-D^{-}\omega^{\pm}\big(x_{\frac{N}{2}}\big)=0.$$\\
	It follows from the discrete minimum principle that $\omega^{\pm}(x_{i})\ge0,~~\forall x_i\in\bar{\Omega}_{N}$, which implies
	$$\|Y\|_{\bar{\Omega}_{N}}\le C$$
\end{proof}
\section{Error estimates}\label{sec5}
Let us denote the nodal error at each mesh point $x_{i}\in\bar{\Omega}_{N}$ by 
\[\lvert e(x_{i})\rvert =\lvert Y(x_{i})-y(x_{i})\rvert ,\]
where $Y$ and $y$ are solutions of equation \eqref{twoparameter} and \eqref{DE} at a point $x_i$ respectively.
 
We find the bounds for the nodal error $\lvert e(x_{i})\rvert $ in $\Omega_{N}^-$ and $\Omega_{N}^+$ separately.
To find the error bounds, we decompose the solution $Y$ of the discrete problem (\ref{DE}) into regular, and layer parts  as \begin{equation}\label{decomposition}
Y(x_{i})=V^{\ast}(x_{i})+W^\ast(x_{i}).\end{equation}
We further split the regular and layer section into parts to the left and right of the discontinuity, i.e., in $\Omega_{N}^-$ and $\Omega_{N}^+$.

Let $V^{*-}(x_{i})$ and $V^{*+}(x_{i})$ be mesh functions, which approximate $V^{\ast}(x_{i})$  to the left and right sides of the point of discontinuity $x_{\frac{N}{2}}=d$  respectively, be defined as follows:
\begin{equation}
\label{regular discrete}V^{*}(x)=\left\{
\begin{array}{ll}
\displaystyle V^{*-}(x_{i}), & \hbox{for} ~1\le i\le \frac{N}{2}-1,\\
\displaystyle V^{*+}(x_{i}), & \hbox{for}~ \frac{N}{2}+1\le i\le N-1,
\end{array}
\right.\end{equation}
where $V^{*-}(x)$ and $V^{*+}(x)$ are, respectively, the solutions to the following discrete problems:
\begin{eqnarray*}
\mathcal{L}^{N}V^{*-}(x_{i})&=&f(x_{i}), \; 1\le i\le \frac{N}{2}-1,\quad V^{*-}(0)=v^{*}(0),\quad V^{*-}(x_{\frac{N}{2}})=v^{*}(d-),\\
\mathcal{L}^{N}V^{*+}(x_{i})&=&f(x_{i}),\quad \frac{N}{2}+1\le i\le N-1,\; V^{*+}(x_{\frac{N}{2}})=v^{*}(d+),\quad V^{*+}(1)=v^{*}(1).
\end{eqnarray*}
 Similarly, we split the mesh function $W^\ast(x_i)$ into left and right layer components $W_{l}^{*}(x_{i})$ and $ W_{r}^{*}(x_{i})$. We further decompose them into components on either side of the discontinuity, $x_{\frac{N}{2}}=d$. 

The decomposition is as follows:
\[W^{*}(x_i)=W_{l}^{*}(x_{i})+W_{r}^{*}(x_{i})=\left\{
\begin{array}{ll}
\displaystyle W_{l}^{*-}(x_i)+W_{r}^{*-}(x_i), & \hbox{for $1\le i\le \frac{N}{2}-1,$ }\\
W_{l}^{*+}(x_{i})+W_{r}^{*+}(x_{i}), & \hbox{for $\frac{N}{2}+1\le i\le N-1,$ }
\end{array}
\right.\]
where $W_{l}^{*-}(x_{i}), W_{l}^{*+}(x_{i})$, $W_{r}^{*-}(x_{i})$ and $ W_{r}^{*+}(x_{i})$ are solutions of the following equations:

\begin{equation}
\label{Wleft}
\begin{aligned}
	\left\{
	\begin{array}{ll} 
\mathcal{L}^{N}W_{l}^{*-}(x_{i})=0,&\; 1\le i\le \frac{N}{2}-1,  W_{l}^{*-}(0)=w_{l}^{*-}(0),   W_{l}^{*-}(x_{\frac{N}{2}})= w_{l}^{*-}(d-),\\
\mathcal{L}^{N}W_{l}^{*+}(x_{i})=0,&\; \frac{N}{2}+1\le i\le N-1,   W_{l}^{*+}(x_{\frac{N}{2}})=w_{l}^{*+}(d+),  W_{l}^{*+}(1)=0,
\end{array}
\right.
\end{aligned}
\end{equation}
\begin{equation}
\label{Wright}
	\begin{aligned}
			\left\{
		\begin{array}{ll} 
\mathcal{L}^{N}W_{r}^{*-}(x_{i})=0,&\; 1\le i\le \frac{N}{2}-1,  W_{r}^{*-}(0)=0,  W_{r}^{*-}(x_{\frac{N}{2}})=w_{r}^{*-}(d-),\\
\mathcal{L}^{N}W_{r}^{*+}(x_{i})=0,&\; \frac{N}{2}+1\le i\le N-1,  W_{r}^{*+}(x_{\frac{N}{2}})=0, W_{r}^{*+}(1)=w_{r}^{*+}(1).
\end{array}
\right.
\end{aligned}
\end{equation}

The unique solution $Y(x_i)$ of the problem (\ref{DE}) is defined by
$$Y(x_i)=\left\{
\begin{array}{ll}
\displaystyle (V^{*-}+W_{l}^{*-}+W_{r}^{*-})(x_{i}), & \hbox{ $1\le i\le \frac{N}{2}-1,$ }\\
\displaystyle  (V^{*-}+W_{l}^{*-}+W_{r}^{*-})(x_{i})=(V^{*+}+W_{l}^{*+}+W_{r}^{*+})(x_{i}), & \hbox{ $i=\frac{N}{2}$,}\\
\displaystyle (V^{*+}+W_{l}^{*+}+W_{r}^{*+})(x_{i}), & \hbox{ $\frac{N}{2}+1\le i\le N-1.$ }
\end{array}
\right.$$
The next lemma gives bounds on the discrete layer components.
\begin{lemma} \label{discrete singular}
	The layer components $W_{l}^{*-}(x_{i}), W_{l}^{*+}(x_{i})$, $W_{r}^{*-}(x_{i})$ and $ W_{r}^{*+}(x_{i})$ satisfy the following bounds:
\begin{align*}
    & \lvert W_{l}^{*-}(x_{i})\rvert \le C\gamma_{l,i}^{-}, \quad  \gamma_{l,i}^{-}=\prod_{k=1}^{i}(1+\theta_{2}h_{k})^{-1}, \; 1\leq i \leq \frac{N}{2}, & \gamma_{l,0}^{-}=C_{1}, \\
    & \lvert W_{l}^{*+}(x_{i})\rvert \le C \gamma_{l,i}^{+}, \quad  \gamma_{l,i}^{+}=\prod_{k=\frac{N}{2}+1}^{i}(1+\theta_{1}h_{k})^{-1}, \; \frac{N}{2}+1 \leq i \leq N, & \gamma_{l,x_{\frac{N}{2}}}^{+}=C_{1},\\
    & \lvert W_{r}^{*-}(x_{i})\rvert \le C \gamma_{r,i}^{-}, \quad \gamma_{r,i}^{-}=\prod_{k=i+1}^{N/2}(1+\theta_{1}h_{k})^{-1},\; 1\leq i \leq \frac{N}{2}, & \gamma_{l,x_{\frac{N}{2}}}^{-}=C_{1},\\
    &\lvert W_{r}^{*-}(x_{i})\rvert \le C \gamma_{r,i}^{+}, \quad  
     \gamma_{r,i}^{+}=C\prod_{k=i+1}^{N}(1+\theta_{2}h_{k})^{-1}, \; \frac{N}{2}+1 \leq i \leq N, & \gamma_{l,N}^{+}=C_{1}.
\end{align*}
\end{lemma}
%%%%%%%%%%%%%%%5proof
\begin{proof}
Define the barrier function for the left layer term as
\[\eta^{-}_{l,i}= \gamma_{l,i}^{-} \pm W_{l}^{*-}(x_{i}),\quad  0\leq i \leq \frac{N}{2}.\]	
For large enough $C$ and $C_1$, $\eta^{-}_{l,0}\geq 0$ and $\eta^{-}_{l,N/2}\geq 0$.

Consider,
 \begin{eqnarray*}
 \mathcal{L}^N \eta^{-}_{l,i}&=& \mathcal{L}^N  \gamma_{l,i}^{-} \pm \mathcal{L}^N W_{l}^{*-}(x_{i})\\
 &=& \gamma_{l,i+1}^{-}\left( 2 \epsilon \theta_{2}^2 \big(\frac{h_{i+1}}{h_{i+1}+{h_i}}-1\big) +2\epsilon \theta_{2}^2 -\mu a(x_i) \theta_2 (1+\theta_{2}h_{i+1})- b(x_i)(1+\theta_{2}h_{i+1}). \right)\\
 &\le& \gamma_{l,i+1}^{-}\left(2\epsilon \theta_{2}^2 -\mu a(x_i) \theta_2 (1+\theta_{2}h_{i+1})- b(x_i)(1+\theta_{2}h_{i+1}) \right)~\text{as } \frac{h_{i+1}}{h_{i+1}+{h_i}}-1\leq  0 
 \end{eqnarray*}
For both the cases $\sqrt{\alpha}\mu\le \sqrt{\rho \epsilon}$ and 
 $\sqrt{\alpha}\mu > \sqrt{\rho \epsilon}$, on simplification, we get
\begin{eqnarray*}
	\mathcal{L}^N \eta^{-}_{l,i} &\le & \gamma_{l,i+1}^{-}\left( 2 \epsilon \theta_{2}^2-\mu a(x_i) \theta_2 -b(x_i) \right) \quad \text{ as } -(\mu a(x_i) \theta_2^2+ b(x_i) \theta_{2}) h_{i+1} \leq 0\\
 &\leq & \gamma_{l,i+1}^{-} \left(\frac{\rho \alpha}{2}+ \frac{\rho \lvert a(x_i)\rvert }{2}-b(x_i) \right) \\
 &\leq& 0.
 \end{eqnarray*}
By discrete minimum principle for the continuous case \cite{RPS}, we obtain  \[\eta^{-}_{l,i} \geq 0 \implies W_{l}^{*-}(x_{i}) \leq C\prod_{k=1}^{i}(1+\theta_{2}h_{k})^{-1}, \; 1\leq i \leq \frac{N}{2}.\]
For $\frac{N}{2}+1 \leq i \leq N,$ consider the barrier function for the left layer term as:
\[\eta^{+}_{l,i}= \gamma_{l,i}^{+} \pm W_{l}^{*+}(x_{i}),\quad   \frac{N}{2} \leq i \leq N.\]	
For large enough $C$ and $C_1$, $\eta^{+}_{l,N/2}\geq 0$ and $\eta^{+}_{l,N}\geq 0$.

Consider 
 \begin{eqnarray*}
	\mathcal{L}^N \eta^{+}_{l,i}&=& \mathcal{L}^N  \gamma_{l,i}^{+} \pm \mathcal{L}^N W_{l}^{*+}(x_{i})\\
	&=& \gamma_{l,i+1}^{+}\left( 2 \epsilon \theta_{1}^2 \big(\frac{h_{i+1}}{h_{i+1}+{h_i}}-1\big) +2\epsilon \theta_{1}^2 -\mu a(x_i) \theta_1 - b(x_i)(1+\theta_{1}h_{i+1})  \right)\\
	&\le& \gamma_{l,i+1}^{+}\left(2\epsilon \theta_{1}^2 -\mu a(x_i) \theta_1- b(x_i)(1+\theta_{1}h_{i+1})  \right)\quad \text{as}~ \frac{h_{i+1}}{h_{i+1}+{h_i}}-1\leq  0 \\
	&\leq & \gamma_{l,i+1}^{+}\left( 2 \epsilon \theta_{1}^2-\mu a(x_i) \theta_1 -b(x_i) \right) \quad (\text{ as } b(x_i) \theta_{1} h_{i+1} \geq 0).
\end{eqnarray*}
For case $\sqrt{\alpha}\mu\le \sqrt{\rho \epsilon},\; \displaystyle \theta_1= \frac{\sqrt{\rho \alpha}}{2 \sqrt{\epsilon}}$, the above expression becomes,
\begin{eqnarray*}
	\mathcal{L}^N \eta^{+}_{l,i}&\leq & \gamma_{l,i+1}^{+} \left(\frac{\rho \alpha}{2}- \mu a(x_i) \frac{ \sqrt{\rho \alpha}}{2 \sqrt{\epsilon}}-b(x_i) \right)\\
	&\leq & \gamma_{l,i+1}^{+} \left(\rho \alpha - b(x_i) - \mu a(x_i) \frac{ \sqrt{\rho \alpha}}{2 \sqrt{\epsilon}} \right) \leq 0.
\end{eqnarray*}
For the case $\sqrt{\alpha}\mu > \sqrt{\rho \epsilon},\; \displaystyle \theta_1= \frac{\mu \alpha}{2 \epsilon}$, we obtain
\begin{eqnarray*}
	\mathcal{L}^N \eta^{+}_{l,i}&\leq & \gamma_{l,i+1}^{+} \left(\frac{\mu^2 \alpha^2}{2\epsilon}- \mu a(x_i) \frac{\mu \alpha}{2 \epsilon}-b(x_i) \right)\\
	&\leq & \gamma_{l,i+1}^{+} \left( - b(x_i)  \right) \leq 0.
\end{eqnarray*}
Hence by discrete minimum principle for continuous case \cite{RPS}, we obtain  \[\eta^{+}_{l,i} \geq 0 \implies W_{l}^{*+}(x_{i}) \leq C\prod_{k=\frac{N}{2}+1}^{i}(1+\theta_{1}h_{k})^{-1}, \; \frac{N}{2}+1 \leq i \leq N.\]
Similarly, we define the barrier function for the right layer component  as
\[\eta^{-}_{r,i}= \gamma_{l,i}^{-} \pm W_{r}^{*-}(x_{i}),\quad  0\leq i \leq \frac{N}{2}.\]	
For large enough $C$ and $C_1$, $\eta^{-}_{r,0}\geq 0$ and $\eta^{-}_{r,N/2}\geq 0$.
Consider,
\begin{eqnarray*}
	\mathcal{L}^N \eta^{-}_{r,i}&=& \mathcal{L}^N  \gamma_{r,i}^{-} \pm \mathcal{L}^N W_{r}^{*-}(x_{i})\\
	&=& \frac{\gamma_{r,i}^{-}}{1+\theta_1h_{i}}\left( 2 \epsilon \theta_{1}^2 \big(\frac{h_{i}}{h_{i+1}+{h_i}}-1\big) +2\epsilon \theta_{1}^2 +\mu a(x_i) \theta_1 - b(x_i)(1+\theta_{1}h_{i})  \right)\\
&\le & \frac{\gamma_{r,i}^{-}}{1+\theta_1h_{i}}\left( 2 \epsilon \theta_{1}^2+\mu a(x_i) \theta_1 -b(x_i) \right) \quad \text{as } \frac{h_{i+1}}{h_{i+1}+{h_i}}-1\leq  0  \text{ and } - b(x_i) \theta_{1} h_{i} \leq 0.
\end{eqnarray*}
For both the cases $\sqrt{\alpha}\mu\le \sqrt{\rho \epsilon}$ and 
 $\sqrt{\alpha}\mu > \sqrt{\rho \epsilon}$, on simplification, we get
\begin{eqnarray*}
	\mathcal{L}^N \eta^{-}_{r,i}&\leq & \frac{\gamma_{r,i}^{-}}{1+\theta_1h_{i}}\left(-b(x_i) \right) \leq 0.
\end{eqnarray*}
By discrete minimum principle for the continuous case \cite{RPS}, we obtain  \[\eta^{-}_{r,i} \geq 0 \implies W_{r}^{*-}(x_{i}) \leq C\prod_{k=i+1}^{N/2}(1+\theta_{2}h_{k})^{-1}, \; 1\leq i \leq \frac{N}{2}.\]
Similarly, we prove the bound for $ W_{r}^{*+}$ for $\frac{N}{2}+1 \leq i \leq N-1.$
\end{proof}	

\begin{lemma} \label{regular part}
	The error in the regular component satisfies the following error estimates for the mesh points, $x_{i}\in {\Omega}_{N}$
	$$\lvert (V^{*}-v^{*})(x_i)\rvert \le CN^{-1},$$
	where $V^{*}$ and $v^{*}$ are the regular part of the continuous and the discrete solution as defined by equations \eqref{regular discrete} and \eqref{regular}, respectively.
\end{lemma}
\begin{proof}
	The truncation error for the regular part of the solution $y$ of the equation \eqref{twoparameter}  for both the cases $\sqrt{\alpha}\mu\le \sqrt{\rho \epsilon}$ and $\sqrt{\alpha}\mu>\sqrt{\rho\epsilon},$ is 
	\begin{align*}
\lvert \mathcal{L}^{N}(V^{*-}-v^{*-})(x_{i})\rvert &=\vert \mathcal{L}^{N}v^{*-}(x_{i})-f(x_{i})\rvert \\
&\le \bigg\lvert \epsilon \bigg(\delta^{2}-\frac{d^{2}}{dx^{2}}\bigg)v^{*-}(x_{i})\bigg\rvert +\mu \lvert a(x_{i})\rvert  \bigg\lvert \bigg(D^{-}-\frac{d}{dx}\bigg)v^{*-}(x_{i})\bigg\rvert \\
&\le CN^{-1},~~~~~ \text{for}~~ 1\le i\le \frac{N}{2}-1.
	\end{align*}
Similarly
\begin{align*}
\lvert \mathcal{L}^{N}(V^{*+}-v^{*+})\rvert \le CN^{-1},~~~~ \text{for}~~ \frac{N}{2}+1\le i\le N-1.
\end{align*}
Define the barrier function
$$\psi^{\pm}(x_{i})=CN^{-1}\pm(V^{*-}-v^{*-})(x_{i}),~~~ 1\le i\le \frac{N}{2}-1.$$
For large C,  $\psi^{\pm}(0)\ge0, \; \psi^{\pm}(x_{\frac{N}{2}})\ge0$ and $\mathcal{L}^N\psi^{\pm}(x_{i})\le0$. 
Hence using the approach given in \cite{FHMRS1}, we get $\psi^{\pm}(x_{i})\ge0$ and 
\begin{equation}
\label{V-}
\lvert (V^{*-}-v^{*-})(x_i)\rvert \le CN^{-1},~~~1\le i\le \frac{N}{2}-1.
\end{equation}
Similarly, 
\begin{equation}
\label{V+}
\lvert (V^{*+}-v^{*+})(x_i)\rvert \le CN^{-1},~~~\frac{N}{2}+1\le i\le N-1.
\end{equation}
Combining the above results, we obtain
\[\lvert (V^{*}-v^{*})(x_i)\rvert \le CN^{-1},\quad \forall~ x_i\in\Omega_{N}.\]
\end{proof}
\begin{lemma} \label{left layer}
The left singular component of the truncation error satisfy the following estimate at mesh point $x_i\in\Omega_{N}$
$$\lvert (W_{l}^{*}-w_{l}^{*})(x_i)\rvert \le  CN^{-1},$$
where $W_l^{*}$ and $w_l^{*}$ are the discrete and the continuous left layer components satisfying the  equations (\ref{Wleft}) and \eqref{left continuous}, respectively.
\end{lemma}
\begin{proof}
	In $[\sigma_{1}, d)$ i.e., for $ \frac{N}{8} \leq i < \frac{N}{2}$, from Theorem (\ref{boundswl}), we obtain

	\begin{equation}
	\label{wl}
	\lvert w_{l}^{*-}(x_{i})\rvert \le C\exp^{-\theta_{2}x_{i}}\le C\exp^{- \theta_{2} \sigma_{1}}\le CN^{-4}.
	\end{equation}
	Also from Lemma (\ref{discrete singular}), we have that $W_{l}^{*-}$ is a monotonically decreasing function, so
	$$ \lvert W_{l}^{*-}(x_{i})\rvert \le C\prod_{k=1}^{\frac{N}{8}}(1+\theta_{2}h_{k})^{-1}, \text{ for } \frac{N}{8} \leq i < \frac{N}{2}.$$
	Now, 	
	\begin{align*}
	\lvert \gamma_{l,\frac{N}{8}}^{-}\rvert &= \prod_{k=1}^{\frac{N}{8}}(1+\theta_{2}h_{k})^{-1}\\
	\implies \log(\gamma_{l,\frac{N}{8}}^{-})&=-\sum_{k=1}^{\frac{N}{8}}\log (1+\theta_{2}h_{k}).	\end{align*}
Consider,
	 \begin{eqnarray*}
	 	\log\bigg(\prod_{k=1}^{\frac{N}{8}} (1+\theta_{2}h_{k})\bigg) &\ge&  \sum_{k=1}^{\frac{N}{8}}\theta_{2}h_{k}-\sum_{k=1}^{\frac{N}{8}}\bigg(\frac{\theta_{2} h_{k}}{2}\bigg)^2, (\text{ as } \log(1+t) \ge t-\frac{t^2}{2}~ \text{for}~ t\geq 0)\\
	 	&=& \theta_2 \sigma_1 - \sum_{k=1}^{\frac{N}{8}}\bigg(\frac{\theta_{2} h_{k}}{2}\bigg)^2
	 	~~\bigg(\text{as} \sum_{k=1}^{\frac{N}{8}}h_{k}=x_{\frac{N}{8}}\bigg).
\end{eqnarray*}
	Next, we calculate $\displaystyle \sum_{k=1}^{\frac{N}{8}}\bigg(\frac{\theta_{2} h_{k}}{2}\bigg)^2$.\\
	For $1\le k\le \frac{N}{8},$\\
	\begin{align*}
	h_{k}=x_{k}-x_{k-1}&=\frac{8}{\theta_{2}}(\phi_{1}(\xi_{k})-\phi_{1}(\xi_{k-1}), \; \xi=\frac{k}{N}\\
	&=\int_{\xi_{k-1}}^{\xi_{k}}\phi_{1}'(\xi)d\xi\\
	\frac{\theta_{2}h_{k}}{8} &= \int_{\xi_{k-1}}^{\xi_{k}}\phi_{1}'(\xi)d\xi \\
	\implies
	\bigg(\frac{\theta_{2}h_{k}}{8}\bigg)^2 &\le(\xi_{k}-\xi_{k-1})\int_{\xi_{k-1}}^{\xi_{k}}\phi_{1}'(\xi)^{2}d\xi, \; \text{ by Holder's inequality}\\
\sum_{k=1}^{\frac{N}{8}}\bigg(\frac{\theta_{2}h_{k}}{8}\bigg)^2 
&\le \sum_{k=1}^{\frac{N}{8}}(\xi_{k}-\xi_{k-1})\int_{\xi_{k-1}}^{\xi_{k}}\phi_{1}'(\xi)^{2}d\xi,\\
&\le N^{-1}\int_{0}^{\frac{1}{8}}\phi_{1}'(\xi)^{2}d\xi\\
&\le C. \quad (\text{from Lemma \ref{assump}})
	\end{align*}
	So
	\begin{align*}
 \lvert \gamma_{l,\frac{N}{8}}^{-}\rvert &\le CN^{-4}\\
		\lvert W_{l}^{*-}(x_{i})\rvert &\le CN^{-4}, \quad  \text{for} \; \frac{N}{8} \leq i < \frac{N}{2}.
	\end{align*}
	Hence for all $x_{i}\in [\sigma_{1},d)$  we have
	$$\lvert (W_{l}^{*-}-w_{l}^{*-})(x_{i})\rvert \le \lvert W_{l}^{*-}(x_{i})\rvert +\lvert w_{l}^{*-}(x_{i})\rvert \le CN^{-4}.$$

 For $\sqrt{\alpha} \mu\le\sqrt{\rho\epsilon} $, the  truncation error for the left layer component  in the inner region $(0,\sigma_{1}),$ i.e., for $i=1,2,\ldots, \frac{N}{8}-1$, is
	\begin{align*}
	\lvert \mathcal{L}^{N}(W_{l}^{*-}-w_{l}^{*-})(x_{i})\rvert &\le C\bigg[\epsilon \int_{x_{i-1}}^{x_{i+1}}\lvert w_{l}^{*-(3)}(x_{i})\rvert dx+\mu \lvert a(x_{i})\rvert \int_{x_{i}}^{x_{i+1}}\lvert w_{l}^{*-(2)}(x_{i})\rvert dx\bigg]\\
	&\le \frac{C}{\sqrt{\epsilon}}\bigg[\int_{x_{i-1}}^{x_{i+1}}e^{-\theta_{2}x}dx+\int_{x_{i}}^{x_{i+1}}e^{-\theta_{2}x}dx\bigg], \quad (\text{from Theorem \ref{boundswl}})  \\
	&\le \frac{C}{\sqrt{\epsilon}}\bigg[\int_{\xi_{i-1}}^{\xi_{i+1}}e^{-8\phi_{1}(\xi)}\frac{\phi_{1}'(\xi)}{\theta_2}d\xi+\int_{\xi_{i}}^{\xi_{i+1}}e^{-8\phi_{1}(\xi)}\frac{\phi_{1}'(\xi)}{\theta_2}d\xi\bigg]\\
	& \hspace{2.5in} \bigg(\text{ as }\; x=\frac{8}{\theta_2}\phi_{1}(\xi)\bigg) \\
	&\leq  \frac{C}{\sqrt{\epsilon}}\bigg[\int_{\xi_{i-1}}^{\xi_{i+1}}e^{-7\phi_{1}(\xi)}\lvert \psi_{1}'(\xi)\rvert d\xi+\int_{\xi_{i}}^{\xi_{i+1}}e^{-7\phi_{1}(\xi)}\lvert \psi_{1}'(\xi)\rvert d\xi\bigg] \\&\le CN^{-1}e^{\frac{-7}{8}\theta_{2}x_{i}}\max\lvert \psi_{1}'\rvert \\&\le CN^{-1}~~(\text{as}~\max\vert \psi_{1}'\rvert \le 8).
	\end{align*}
	We choose the barrier function for the layer component as
	$$\psi^{\pm}(x_{i})=CN^{-1}\pm(W_{l}^{*-}-w_{l}^{*-})(x_{i}),\; i=1,2,\ldots, \frac{N}{8}-1.$$
	For sufficiently large $C$, we have $\mathcal{L}^{N}\psi_{i}\le0$.  Hence by discrete maximum principle in \cite{RPS}, $\psi_{i}\ge0$. So, by the comparison principle, we can obtain the following bounds:
	$$\lvert (W_{l}^{*-}-w_{l}^{*-})(x_{i})\rvert \le CN^{-1} \,\quad  \forall \; 1\leq i \leq \frac{N}{8}-1.$$
	For $\sqrt{\alpha} \mu>\sqrt{\rho\epsilon}$, the  truncation error for the left layer component  for $i=1,2,\ldots, \frac{N}{8}-1$ is given by
	\begin{align*}
	\lvert  \mathcal{L}^{N}(W_{l}^{*-}-w_{l}^{*-})(x_{i})\rvert &\le C\bigg(\epsilon \int_{x_{i-1}}^{x_{i+1}}\lvert w_{l}^{*-(3)}(x_{i})\rvert dx+\mu \lvert a(x_{i})\rvert \int_{x_{i}}^{x_{i+1}}\lvert w_{l}^{*-(2)}(x_{i})\rvert dx\bigg)\\
	&\le \frac{C_{1}\epsilon}{\mu^3}\bigg[\int_{x_{i-1}}^{x_{i+1}}e^{-\theta_{2}x}\bigg]dx+\frac{C_{2}}{\mu}\bigg[\int_{x_{i}}^{x_{i+1}}e^{-\theta_{2}x}dx\bigg] \quad (\text{using Theorem \ref{wr}})\\&\le C\bigg[\int_{\xi_{i-1}}^{\xi_{i+1}}e^{-7\phi_{1}(\xi)}\lvert \psi_{1}'(\xi)\rvert d\xi+\int_{\xi_{i}}^{\xi_{i+1}}e^{-7\phi_{1}(\xi)}\lvert \psi_{1}'(\xi)\rvert d\xi\bigg] \\ 
	&\le CN^{-1} \max\lvert \psi_{1}'\rvert  \leq CN^{-1}\quad (\text{as}~\max\lvert \psi_{1}'\rvert \le 8).
	\end{align*}
	Choosing a barrier function for the layer component as
	$$\psi^{\pm}(x_{i})=CN^{-1}\pm(W_{l}^{*-}-w_{l}^{*-})(x_{i}), \; \forall \; 1\leq i \leq \frac{N}{8}-1.$$
For sufficiently large $C$, we have  $\mathcal{L}^{N}\psi_{i}\le0$. Using the discrete minimum principle in \cite{RPS},    we can obtain the following bounds:
	$$\lvert (W_{l}^{*-}-w_{l}^{*-})(x_{i})\rvert \le CN^{-1},  \; \forall \; 1\leq i \leq \frac{N}{8}-1.$$
	Hence for the left layer component
	\begin{equation}
	\label{w-}
	\lvert (W_{l}^{*-}-w_{l}^{*-})(x_i)\rvert \le CN^{-1},\quad \forall \; 1\leq i \leq \frac{N}{2}-1. 
		\end{equation}
By similar argument in the domains  $(d, 1-\sigma_4]$ and $(1-\sigma_4,1)$, we have 
	\begin{equation}
	\label{w+}
	\lvert (W_{l}^{*+}-w_{l}^{*+})(x_i)\rvert \le CN^{-1}, \quad \forall \; \frac{N}{2}+1 \leq i \leq N-1.
	\end{equation}
	Combining the results (\ref{w-}) and (\ref{w+}), the desired result is obtained.
\end{proof}
\begin{lemma} \label{right layer}
	The right singular component of the truncation error satisfies the following approximation for each mesh point, $x_{i}\in\Omega_{N}$
\[	\lvert (W_{r}^{*}-w_{r}^{*})(x_i)\rvert \le 	\displaystyle CN^{-1},\]
	where $W_r^{*}$ and $w_r^{*}$ are the  discrete and the continuous right layer components satisfying the  equations (\ref{Wright}) and (\ref{right continuous}), respectively.
\end{lemma}
\begin{proof}
		In $(0, d-\sigma_{2}]$, for $1\le i\le \frac{3N}{8}$, the left layer component has the following bound from Theorem (\ref{wr})
	
	\begin{equation}
	\label{wR}
	\lvert w_{r}^{*-}(x_{i})\rvert \le Ce^{-\theta_{1}(d-x_{i})}\le Ce^{-\theta_{1}\sigma_{2}}\le CN^{-4}.
	\end{equation}
	Also from Lemma (\ref{discrete singular}), we see that $W_{r}^{*-}$ is increasing function. So
	\[ \lvert W_{r}^{*-}(x_{i})\rvert \le C\prod_{j=i+1}^{\frac{N}{2}}(1+\theta_{1}h_{j})^{-1}\le C \lvert \gamma_{r,\frac{3N}{8}}^{-}\rvert  \; \text{for }1\le i\le \frac{3N}{8}.
\]
	Now consider, 	
\begin{eqnarray*}
\lvert \gamma_{r,\frac{3N}{8}}^{-}\rvert  &=& \prod_{j=\frac{3N}{8}+1}^{\frac{N}{2}}(1+\theta_{1}h_{j})^{-1}\\
 \log(\gamma_{r,\frac{3N}{8}}^{-})&=&-\sum_{j=\frac{3N}{8}+1}^{\frac{N}{2}} \log (1+\theta_{1}h_{j})\\ 
 \text{As}~ \log(1+t^2)&\ge& t-\frac{t^2}{2}~ \text{for}~ t\ge 0,	\\
 \implies  \sum_{j=\frac{3N}{8}+1}^{\frac{N}{2}} \log (1+\theta_{1}h_{j})& \ge& \sum_{j=\frac{3N}{8}+1}^{\frac{N}{2}}\theta_{1}h_{j}-\sum_{k=\frac{3N}{8}+1}^{\frac{N}{2}}\bigg(\frac{\theta_{1} h_{j}}{2}\bigg)^2, \quad \bigg(\text{as} \sum_{j=\frac{3N}{8}+1}^{\frac{N}{2}}h_{j}=x_{\frac{N}{2}}\bigg).
	\end{eqnarray*}
	Now we calculate $\sum_{j=\frac{3N}{8}+1}^{\frac{N}{2}}\bigg(\frac{\theta_{1} h_{j}}{2}\bigg)^2$.\\
	For $\frac{3N}{8}+1 \le j\le \frac{N}{2},$\\
	\begin{eqnarray*}
	h_{j}&=&x_{j}-x_{j-1}=d-\frac{8}{\theta_{1}}\phi_{2}(\xi_{j})-\left(d-\frac{8}{\theta_{1}}\phi_{2}(\xi_{j-1})\right)\\
	&=&\frac{8}{\theta_{1}}(\phi_{2}(\xi_{j})-\phi_{2}(\xi_{j-1})) =\frac{8}{\theta_{1}}\int_{\xi_{j-1}}^{\xi_{j}}\phi_{2}'(\xi)d\xi\\
	\frac{\theta_1 h_j}{8}&=& \int_{\xi_{j-1}}^{\xi_{j}}\phi_{2}'(\xi)d\xi\\
	\bigg(\frac{\theta_{1}h_{j}}{8}\bigg)^2 &\le&(\xi_{j}-\xi_{j-1})\int_{\xi_{j-1}}^{\xi_{j}}\phi_{2}'(\xi)^{2}d\xi \quad \text{(by Holder's inequality)} \\
	\sum_{j=\frac{3N}{8}+1}^{\frac{N}{2}}\bigg(\frac{\theta_{1}h_{j}}{8}\bigg)^2&\le& N^{-1}\int_{\frac{3}{8}}^{\frac{1}{2}}\phi_{2}'(\xi)^{2}d\xi \le C \quad \text{(from Lemma \ref{assump})}.
	\end{eqnarray*}
	So
	\begin{align*}
	\sum_{j=\frac{3N}{8}+1}^{\frac{N}{2}} \log (1+\theta_{1}h_{j})& \ge 4\log N-C\\
%	\log(W_{r}^{*-}(x_{\frac{3N}{8}}))&=-\sum_{j=\frac{3N}{8}+1}^{\frac{N}{2}}(1+\theta_{1}h_{j})\\&\le -4\log N+C\\
\implies \lvert W_{r}^{*-}(x_{i})\rvert &\le C \gamma_{r,\frac{3N}{8}}^{-}\le CN^{-4}, \quad \forall \; 1\leq i \leq \frac{3N}{8}.
	\end{align*}
	Hence for all $x_{i}\in (0,d-\sigma_{2}]$, we have
	\[\lvert (W_{r}^{*-}-w_{r}^{*-})(x_{i})\lvert \le \lvert W_{r}^{*-}(x_{i})\rvert +\lvert w_{r}^{*-}(x_{i})\rvert \le CN^{-4}.\]
	For $\sqrt{\alpha} \mu\le\sqrt{\rho\epsilon}$, the derivative bounds for right layer component $w_{r}^{*-}$ in the inner region $(d-\sigma_{2}, d)$ is given by Theorem (\ref{boundswl}).
	Truncation error for right layer component is given by,
	\begin{align*}
	\lvert \mathcal{L}^{N}(W_{r}^{*-}-w_{r}^{*-})(x_{i})\rvert &\le C\bigg(\epsilon \int_{x_{i-1}}^{x_{i+1}}\lvert w_{r}^{*-(3)}(x_{i})\rvert dx+\mu \lvert a(x_{i})\rvert \int_{x_{i}}^{x_{i+1}}\lvert w_{r}^{*-(2)}(x_{i})\rvert dx\bigg)\\&\le C\bigg(\epsilon \int_{x_{i-1}}^{x_{i+1}}\lvert w_{r}^{*-(3)}(x_{i})\rvert dx+\mu \lvert a(x_{i})\rvert \int_{x_{i}}^{x_{i+1}}\lvert w_{r}^{*-(2)}\rvert (x_{i})dx\bigg)\\
	&\le \frac{C}{\sqrt{\epsilon}}\bigg[\int_{x_{i-1}}^{x_{i+1}}e^{-\theta_{1}(d-x)}dx+\int_{x_{i}}^{x_{i+1}}e^{-\theta_{1}(d-x)}dx\bigg] \\
	&\le \frac{C}{\sqrt{\epsilon}}\bigg[\int_{\xi_{i-1}}^{\xi_{i+1}}e^{-7\phi_{2}(\xi)}\lvert \psi_{2}'(\xi)\rvert d\xi+\int_{\xi_{i}}^{\xi_{i+1}}e^{-7\phi_{2}(\xi)}\lvert \psi_{2}'(\xi)\rvert d\xi\bigg] \\
	&\le CN^{-1}e^{\frac{-7}{8}\theta_{1}(d-x_{i})}\max\lvert \psi_{2}^{'}\rvert \\&\le CN^{-1}~~(\text{as}~\max\lvert \psi_{2}^{'}\rvert \le 8).
	\end{align*}
	By defining an appropriate barrier function and using the discrete minimum principle (in  \cite{RPS}), we can obtain the following bounds:
	$$\lvert (W_{r}^{*-}-w_{r}^{*-})(x_{i})\rvert \le CN^{-1}, \quad \frac{3N}{8} < i < \frac{N}{2}.$$
	For case $\sqrt{\alpha} \mu>\sqrt{\rho\epsilon}$, the derivative bounds for right layer component $w_{r}^{*-}$ for $\frac{3N}{8} < i < \frac{N}{2}$ are given by Theorem (\ref{wr}).
	Hence by using truncation error for the right layer component, we obtain,
	\begin{align*}
	\lvert \mathcal{L}^{N}(W_{r}^{*-}-w_{r}^{*-})(x_{i})\rvert &\le C\bigg(\epsilon \int_{x_{i-1}}^{x_{i+1}}\lvert w_{r}^{*-(3)}(x_{i})\rvert dx + \mu \lvert a(x_{i})\rvert  \int_{x_{i}}^{x_{i+1}} \lvert w_{r}^{*-(2)}(x_{i})\rvert dx\bigg)\\
	&\le C_{1}\bigg(\int_{x_{i-1}}^{x_{i+1}}\left(\frac{\mu}{\epsilon}\right)^{3}e^{-\theta_{1}(d-x)}\bigg) dx+C_{2}\bigg(\int_{x_{i}}^{x_{i+1}}\left(\frac{\mu}{\epsilon}\right)^{2}e^{-\theta_{1}(d-x)}dx\bigg)  \\
	&\le  C\frac{\mu}{\epsilon}^{2}\bigg(\int_{\xi_{i-1}}^{\xi_{i+1}}e^{-7\phi_{2}(\xi)}\lvert \psi_{2}'(\xi)\rvert d\xi+\int_{\xi_{i}}^{\xi_{i+1}}e^{-7\phi_{2}(\xi)}\lvert \psi_{2}'(\xi)\rvert d\xi\bigg) \\
	&\le \displaystyle C\frac{\mu}{\epsilon}^{2}e^{\frac{-7}{8}\theta_{1}(d-x_{i})}N^{-1} \max\lvert \psi_{2}'\rvert \\
	&\le C\frac{\mu}{\epsilon}^{2}N^{-1} \quad (\text{as}~\max\lvert \psi_{2}'\rvert \le 8).
	\end{align*}
	Choosing the barrier function for the layer component as
	\[\psi^{\pm}(x_{i})=C_{1}N^{-1}+C_{2}N^{-1}\left(\frac{\mu}{\epsilon}\right)\frac{x_{i}}{\theta_{1}}\pm(W_{r}^{*-}-w_{r}^{*-})(x_{i}).\]
For sufficiently large $C$, by the application of the discrete minimum principle (in \cite{RPS}) we obtain the following bounds:
	\begin{equation*}
	\lvert (W_{r}^{*-}-w_{r}^{*-})(x_{i})\rvert \le C_{1}N^{-1}+C_{2}N^{-1}\left(\frac{\mu}{\epsilon}\right)\frac{x_{i}}{\theta_{1}} \le CN^{-1}, \quad \text{for} \; \frac{3N}{8} < i < \frac{N}{2}.
	\end{equation*}	Hence the bound for the right layer component for $x_i \in (d-\sigma_2, d)$ is 
	\begin{equation}
	\label{wr-}
	\lvert (W_{r}^{*-}-w_{r}^{*-})(x_i)\rvert \le  CN^{-1},
	\end{equation}
	Similarly, we can prove the result for $\frac{N}{2}+1\le i\le N$,
	\begin{equation}
	\label{wr+}
	\lvert (W_{r}^{*+}-w_{r}^{*+})(x_i)\rvert \le  CN^{-1},
	\end{equation}
	Combining the results (\ref{wr-}) and (\ref{wr+}) the final answer is obtained.
\end{proof}
\begin{lemma} \label{discontnuity}
	Let $y(x)$ and $Y(x)$ be the solutions to the problems (\ref{twoparameter}) and (\ref{DE}), respectively. The error $e\left(x_{\frac{N}{2}}\right)$ estimated at the point of discontinuity $x_{\frac{N}{2}}=d$ satisfies the following estimate
	$$\displaystyle \lvert (D^{+}-D^{-})(Y(x_{\frac{N}{2}})-y(x_{\frac{N}{2}}))\rvert \le \left\{
	\begin{array}{ll}
	\displaystyle \frac{C}{\epsilon \theta_{1} }, & \hbox{ $\sqrt{\alpha}\mu \le \sqrt{\rho\epsilon}$, } \vspace{.5mm}\\
\displaystyle	\frac{C \mu^{2}}{\epsilon^{2} \theta_{1}}, & \hbox{ $\sqrt{\alpha}\mu > \sqrt{\rho\epsilon}$.}
	\end{array}
	\right. $$
\end{lemma}
\begin{proof}
	Consider 
	$$\lvert (D^{+}-D^{-})(Y(x_{\frac{N}{2}})-y(x_{\frac{N}{2}}))\rvert \le\lvert (D^{+}-D^{-})y(x_{\frac{N}{2}}))\rvert $$
	Since $\lvert (D^{+}-D^{-})Y(x_{\frac{N}{2}}))\rvert =0$
	\begin{align*}
	\lvert (D^{+}-D^{-})(Y(x_{\frac{N}{2}})-y(x_{\frac{N}{2}}))\rvert &\le\bigg\lvert \bigg(\frac{d}{dx}-D^{+}\bigg)y(x_{\frac{N}{2}})\bigg\rvert +\bigg\lvert \bigg(\frac{d}{dx}-D^{-}\bigg)y(x_{\frac{N}{2}})\bigg\rvert \\&\le C_{1}h_{\frac{N}{2}+1}\lvert y''\rvert +C_{2}h_{\frac{N}{2}}\lvert y''\rvert \\&\le C\bar{h}\lvert y''\rvert \\&\le\left\{
	\begin{array}{ll}
	\displaystyle \frac{C\bar{h}}{\epsilon}, & \hbox{ $\sqrt{\alpha}\mu \le \sqrt{\rho\epsilon}$, ~$(\bar{h}=\max\{h_{\frac{N}{2}},h_{\frac{N}{2}+1}\})$} \vspace{.7mm} \\
\displaystyle	\frac{C\bar{h} \mu^{2}}{\epsilon^{2}}, & \hbox{ $\sqrt{\alpha}\mu > \sqrt{\rho\epsilon}$.}
	\end{array}
	\right. 
	\end{align*}
	Using the fact that $\bar{h} \leq C/\theta_{1}$ in the given domain gives the lemma.
\end{proof}
%%%%%%%%%%%%%%%%%%%
\begin{theorem}
	\label{main}
	Let us assume $\sqrt{\epsilon}<N^{-1}$ for $\sqrt{\alpha}\mu \le \sqrt{\rho\epsilon}$ and  $\displaystyle \max\bigg\{\frac{\epsilon}{\mu},\mu, \bigg\}<N^{-1}$ for $\sqrt{\alpha}\mu > \sqrt{\rho\epsilon}$. Let $y(x)$ and $Y(x)$ be respectively the solutions of the problems (\ref{twoparameter}) and (\ref{DE}) then,
	$$\displaystyle \|Y-y\|\le  CN^{-1}, $$
	where C is a constant independent of $\epsilon, \mu$ and discretization parameter $N$. 
\end{theorem}
\begin{proof} For $i=1,2,\ldots,N/2-1,N/2+1, \ldots, N-1$, from Lemma (\ref{regular part}), Lemma (\ref{left layer}), and Lemma (\ref{right layer}), we have that
	 	$$\|Y-y\|_{\Omega_{N}^- \cup \Omega_{N}^+}\le  CN^{-1}, $$
	 Let $\sqrt{\alpha}\mu \le \sqrt{\rho\epsilon},$ to find error at the point of discontinuity $x_{i}=x_{\frac{N}{2}}$, consider the discrete barrier function $\phi_{1}(x_{i})=\psi_{1}(x_{i})\pm e(x_{i})$ defined in the interval $(d-\sigma_{2},d+\sigma_{3})$ where
	 $$\psi_{1}(x_{i})=CN^{-1}+\frac{C_{1}\sigma}{\epsilon N (\log N)^{2}}\left\{
	 \begin{array}{ll}
	 	\displaystyle x_{i}-(d-\sigma_{2}), & \hbox{ $x_{i}\in \Omega_{N}\cap(d-\sigma_{2}, d] $, }\\
	 	d+\sigma_{3}-x_{i}, & \hbox{ $x_{i}\in\Omega_{N}\cap[d, d+\sigma_{3})$}
	 \end{array}
	 \right. $$
	 and $\displaystyle \sigma=\sigma_{2}=\sigma_{3}=  \frac{4}{\theta_{1}} \log N.$
	 
	 We have $\phi_{1}(d-\sigma_{2})$ and $\phi_{1}(d+\sigma_{3})$ are non-negative. And
	 $\mathcal{L}^{N}\phi_{1}(x_{i})\le 0,~~x_{i}\in( d-\sigma_2, d+\sigma_3),~~\lvert (D^{+}-D^{-})\phi_{1}(x_{\frac{N}{2}}))\rvert \le0.$\\
	Hence by applying discrete minimum priciple we get $\phi_{1}(x_{i})\ge0.$\\
	Therefore, for $x_{i}\in (d-\sigma_{2},d+\sigma_{3})$
	\begin{equation}
	\label{Yd}
	\lvert (Y-y)(x_{i})\rvert \le C_{1}N^{-1}+\frac{C_{2}\sigma^{2}}{\epsilon N(\log N)^{2}}\le CN^{-1}.
	\end{equation}
	In second case $\sqrt{\alpha}\mu > \sqrt{\rho\epsilon}$, consider the discrete barrier function $\phi_{2}(x_{i})=\psi_{2}(x_{i})\pm e(x_{i})$ defined in the interval $(d-\sigma_{2},d+\sigma_{3})$ where
	$$\psi_{2}(x_{i})=CN^{-1}+\frac{C_{1}\sigma\mu^{2}}{\epsilon^{2} N (\log N)^{2}}\left\{
	\begin{array}{ll}
	\displaystyle x_{i}-(d-\sigma_{2}), & \hbox{ $x_{i}\in \Omega_{N}\cap(d-\sigma_{2}, d] $, }\\
	d+\sigma_{3}-x_{i}, & \hbox{ $x_{i}\in\Omega_{N}\cap[d, d+\sigma_{3})$}
	\end{array}
	\right. $$
	where $\displaystyle \sigma=\sigma_{2}= \frac{4}{\theta_{1}} \log N.$
	We have $\phi_{2}(d-\sigma_{2})$ and $\phi_{2}(d+\sigma_{3})$ are non negative and
	 $\mathcal{L}^{N}\phi_{2}(x_{i})\le 0,~~x_{i}\in(d-\sigma_{2},d+\sigma_{3})$,~ and $\lvert (D^{+}-D^{-})\phi_{2}(x_{\frac{N}{2}}))\rvert \le0.$\\
	Hence by applying discrete minimum principle, we get $\phi_{2}(x_{i})\ge0.$
	Therefore, for $x_{i}\in (d-\sigma_{2},d+\sigma_{3})$
	\begin{equation}
	\label{Yd1}
	\lvert (Y-y)(x_{i})\rvert \le C_{1}N^{-1}+\frac{C_{2}\sigma^{2}\mu^{2}}{\epsilon^{2} N(\log N)^{2}}\le CN^{-1}.
	\end{equation}
	By combining the result (\ref{Yd}) and (\ref{Yd1}) we obtain the desired result.
\end{proof}
\section{Numerical results}\label{sec6}
In this section, we have considered some singularly perturbed two-parameter boundary value problems with discontinuous convection coefficient and source term as test problems. The proposed scheme is used to solve these problems numerically.
\begin{Example}
	\label{ex-a}
	$$\epsilon y''(x)+\mu a(x)y'(x)-y(x)=f(x)~~ ~~x\in(0,.5)\cup(0.5,1),$$
	$$	y(0)=2,~ y(1)=1,$$\nonumber\
	with \\
	$$a(x)=\left\{
	\begin{array}{ll}
	\displaystyle -2, & \hbox{ $0\le x\le 0.5$, }\\
	2, & \hbox{ $0.5<x\le1$,}
	\end{array}
	\right. and~~
	f(x)=\left\{
	\begin{array}{ll}
	\displaystyle -1, & \hbox{ $0\le x\le 0.5$, }\\
	1, & \hbox{ $0.5<x\le1$.}
	\end{array}
	\right.$$\\
\end{Example}
\begin{Example}
\label{ex-b}
	$$\epsilon y''(x)+\mu a(x)y'(x)-2y(x)=f(x)~~ ~~x\in(0,.5)\cup(0.5,1),$$
	$$	y(0)=0,~ y(1)=-1,$$\nonumber\
	\text{with} \\
	$$a(x)=\left\{
	\begin{array}{ll}
	\displaystyle -(1+x), & \hbox{ $0\le x\le 0.5$, }\\
	(2+x^{2}), & \hbox{ $0.5<x\le1$,}
	\end{array}
	\right. and~~
	f(x)=\left\{
	\begin{array}{ll}
	\displaystyle -(14x+1), & \hbox{ $0\le x\le 0.5$, }\\
	(2-2x), & \hbox{ $0.5<x\le1$.}
	\end{array}
	\right.$$\\
\end{Example}
Since the exact solution for  Example \ref{ex-a} and  Example \ref{ex-b} is unknown, the maximum point-wise error and rate of convergence are computed using the double mesh principle (see \cite{DM}, page 199). The double mesh difference is defined by
$$E^{N}= \max\limits_{x\in\bar{\Omega}^{N}}\lvert Y^{N}(x_{i})-Y^{2N}(x_{i})\rvert $$
where $Y^{N}(x_{i})$ and $Y^{2N}(x_{i})$ represent the numerical solutions determined using $N$ and $2N$ mesh points respectively.
The numerical rate of convergence is given by
$$R^{N}=\frac{\log(E^{N})-\log(E^{2N})}{\log2}.$$

Table \ref{eg11} shows the results for various values of $\mu$ and for $\epsilon= 10^{-6}$ for Example \ref{ex-a}. The order of convergence obtained approaches one as we increase the number of mesh points. In Table \ref{eg12} the maximum point-wise error and order of convergence are given for Example \ref{ex-a} for varying values of $\epsilon$ and keeping the value of $\mu$ fixed.

Figures \ref{fig11} and \ref{fig12} represent the numerical solution and maximum point-wise error for Example \ref{ex-a} for the case $\sqrt{\alpha}\mu \le \sqrt{\rho\epsilon} $ respectively with $\epsilon=10^{-8}, \mu = 10^{-6} $ and $N=256$.
The numerical solution and maximum point-wise error for the case $\sqrt{\alpha}\mu > \sqrt{\rho\epsilon}$ for Example \ref{ex-a} for $N=256$ is given in Figures \ref{fig13} and \ref{fig14} respectively with $\epsilon=10^{-12}, \mu = 10^{-4}$ and $N=256$.

 In Tables \ref{eg21} and \ref{eg22}, maximum point-wise error and order of convergence are tabulated for Example \ref{ex-b}. From these tables, we observe that the numerical order of convergence is consistent with the theoretical estimates presented in this paper.
 
 For Example \ref{ex-b}, Figures \ref{fig21} and \ref{fig22} gives the numerical solution and maximum point-wise error for the case $\sqrt{\alpha}\mu \le \sqrt{\rho\epsilon}$ respectively with $\epsilon=10^{-8}, \mu = 10^{-6} $ and $N=256$. The Figures  \ref{fig23} and \ref{fig24} show the numerical solution and maximum point-wise error for the case $\sqrt{\alpha}\mu > \sqrt{\rho\epsilon}$ respectively with $\epsilon=10^{-12}, \mu = 10^{-4}$ and $N=256$. From these figures, we observe that the maximum error is occurring at the point of discontinuity.

With the use of the Shishkin-Bakhvalov mesh, we are able to improve the order of convergence to one, unlike the Shishkin mesh, where the order of convergence is deteriorated due to the presence of a logarithmic factor. In Table \ref{comp}, we have compared the order of convergence obtained for the numerical method presented here on the Shishkin-Bakhvalov mesh and Shishkin mesh for Example \ref{ex-a}.
%\section{Tables}\label{sec7}
\begin{table}[h]
\begin{center}
\begin{minipage}{174pt}
\caption{Maximum point-wise error $E^{N}$ and approximate orders of convergence $R^{N}$ for Example \ref{ex-a} when $\epsilon=10^{-6}.$}\label{eg11}%
\begin{tabular}{@{}llllll@{}}
\hline
\multirow{2}{*}{$\mu$ }& \multicolumn{5}{c}{Number of mesh points N}\\
\cline{2-6}
& 64  & 128 & 256 &512 & 1024\\
\hline
		$10^{-4}$ & 3.3161e-01&	2.1205e-01&	1.2184e-01&	6.5947e-02&	3.4499e-02 \\ 
		Order & 0.64507&	0.79940&	0.88563&	0.93474& \\

		$10^{-5}$ & 3.0199e-01&1.8183e-01&	9.9546e-02&	5.2296e-02&	2.6915e-02 \\ 
		
		Order & 0.73190&	0.86918&	0.92864&	0.95830& \\
		$10^{-6}$ & 2.9894e-01&	1.7875e-01&	9.7305e-02&	5.0937e-02&	2.6164e-02 \\ 
	
		Order & 0.74189&	0.87739&	0.93378&	0.96113& \\
		$10^{-7}$ & 2.9863e-01&	1.7844e-01&	9.7080e-02&	5.0801e-02&	2.6089e-02 \\ 
		
		Order & 0.74290&	0.87823&	0.93430&	0.96142& \\
		
		$10^{-8}$ &2.9860e-01&	1.7841e-01&	9.7058e-02	&5.0788e-02	&2.6081e-02\\
		
		Order & 0.74300&	0.87831&	0.93435&	0.96145& \\
		$10^{-9}$ &2.9860e-01&1.7841e-01&9.7056e-02&	5.0787e-02&	2.6081e-02 \\ 
		
		Order & 0.74301&	0.87832&	0.93436&	0.96145& \\
		
		$10^{-10}$ & 2.9860e-01	&1.7841e-01	&9.7056e-02	&5.0786e-02&2.6081e-02\\ 
	
		Order  & 0.74301&	0.87832&	0.93436&	0.96145& \\
		$10^{-11}$ &2.9860e-01&	1.7841e-01&	9.7056e-02&	5.0786e-02&	2.6081e-02\\ 
		
		Order  &0.74301&	0.87832&	0.93436&	0.96145&\\
		
		$10^{-12}$ &2.9860e-01&	1.7841e-01&	9.7056e-02&	5.0786e-02&	2.6081e-02\\ 
	
		Order  &0.74301&	0.87832&	0.93436&	0.96145&\\
		$10^{-13}$ &2.9860e-01&	1.7841e-01&	9.7056e-02&	5.0786e-02&	2.6081e-02\\ 
		
		Order  &0.74301&	0.87832&	0.93436&	0.96145&\\
		$10^{-14}$ &2.9860e-01&	1.7841e-01&	9.7056e-02&	5.0786e-02&	2.6081e-02\\ 
		
		Order  &0.74301&	0.87832&	0.93436&	0.96145&\\
		$10^{-15}$ & 2.9894e-01&	1.7875e-01&	9.7305e-02&	5.0937e-02&	2.6164e-02 \\ 
		
		Order & 0.74189&	0.87739&	0.93378&	0.96113& \\
		$10^{-16}$ &2.9860e-01&	1.7841e-01&	9.7056e-02&	5.0786e-02&	2.6081e-02\\ 
		
		Order  &0.74301&	0.87832&	0.93436&	0.96145&\\
		$10^{-17}$ &2.9860e-01&	1.7841e-01&	9.7056e-02&	5.0786e-02&	2.6081e-02\\ 
		
		Order  &0.74301&	0.87832&	0.93436&	0.96145&\\
\hline
\end{tabular}
\end{minipage}
\end{center}
\end{table}
\begin{table}[h]
\begin{center}
\begin{minipage}{174pt}
\caption{Maximum point-wise error $E^{N}$ and approximate orders of convergence $R^{N}$ for Example \ref{ex-a} when $\mu=10^{-4}.$}\label{eg12}%
\begin{tabular}{@{}llllll@{}}
\hline 
\multirow{2}{*}{$\epsilon$ }& \multicolumn{5}{c}{Number of mesh points N}\\
\cline{2-6}
& 64  & 128 & 256 &512 & 1024\\
\hline
		$10^{-8}$ &	4.3793e-01&3.0942e-01&1.9296e-01&1.0991e-01&	5.9142e-02 \\ 
		
		Order &	0.50113&	0.68126&	0.81198&	0.89406&\\
		$10^{-9}$ &	4.4915e-01&	3.0223e-01&	1.8274e-01&	1.0226e-01&	5.4505e-02 \\ 
		
		Order &	0.57151&	0.72586&	0.83750&	0.90783&\\

		$10^{-10}$	&4.5302e-01	&3.0188e-01&1.8160e-01&	1.0136e-01	&5.3951e-02 \\
		
		Order &	0.58557&	0.73318&	0.84130&	0.90978&\\
		$10^{-11}$	&4.5349e-01&3.0186&	1.8149e-01&	1.0127e-01&	5.3895e-02 \\
		
		Order &0.58716&	0.73398&	0.84170&	0.90998&\\
		
		$10^{-12}$&4.5353e-01& 3.0186e-01&	1.8148e-01	&1.0126e-01&5.3889e-02\\ 
		
		Order &	0.58732&	0.73405&	0.84174&	0.91000& \\
		$10^{-13}$&4.5354e-01&	3.0186e-01&	1.8148e-01&	1.0126&	5.3888e-02\\ 
		
		Order &	0.58734&	0.73406&	0.84175&	0.91001& \\
		
		$10^{-14}$&4.5354e-01&3.0186e-01&	1.8147e-01&	1.0126e-01&	5.3863e-02\\
		
		Order 	&0.58733&	0.73409&	0.84170&	0.91073& \\
		$10^{-15}$&4.5355e-01&	3.0182e-01&	1.8147e-01&	1.0114e-01&	5.3746e-02\\
		
		Order 	&0.58753&	0.73396&	0.84334&	0.91218& \\
		$10^{-16}$&4.5347e-01&3.0154e-01&1.8095e-01&	1.0082e-01&	5.0923e-02\\
		
		Order 	&0.58862&	0.73672&	0.84375&	0.98551& \\
		$10^{-17}$&4.5270e-01&	2.9936e-01&	1.7931e-01&	9.3085e-01&	2.7602e-02\\
		
		Order 	&0.59666&	0.73941&	0.94586&	1.7537& \\
	
\hline
\end{tabular}
\end{minipage}
\end{center}
\end{table}
\begin{table}[h]
\begin{center}
\begin{minipage}{174pt}
\caption{Maximum point-wise error $E^{N}$ and approximate orders of convergence $R^{N}$ for Example \ref{ex-b} when $\epsilon=10^{-6}.$}\label{eg21}%
\begin{tabular}{@{}llllll@{}}
\hline
\multirow{2}{*}{$\mu$ }& \multicolumn{5}{c}{Number of mesh points N}\\
\cline{2-6}
& 64  & 128 & 256 &512 & 1024\\
\hline
			$10^{-4}$ &5.3686e-01&3.4621e-01&1.2697e-01&	4.6968e-02&	2.4601e-02\\ 
		
		Order &0.63289&	0.12697&	1.4470&	1.4348&\\
		$10^{-5}$ &5.5069e-01&	3.7896e-01&	1.5849e-01&	4.7313e-02&	1.0219e-02\\ 
		
		Order &0.53919&	1.2576&	1.7440&	2.2109&\\
		$10^{-6}$ &5.5215e-01&3.8238e-01&1.6172e-01&	4.9611e-02&1.1616e-02\\ 
		
		Order &0.53006&	1.2414&	1.7047&	2.0945&\\
		$10^{-7}$ &5.5230e-01&	3.8272e-01&	1.6204e-01&	4.9840e-02&	1.1755e-02\\ 
		
		Order &0.52915&	1.2398&	1.7010&	2.0839&\\
		
		$10^{-8}$	&5.5231e-01&3.8275e-01&1.6207e-01&	4.9863e-02&	1.1769e-02\\
		
		Order &	0.5290&	1.2397&	1.7006&	2.0829&\\
		$10^{-9}$	&5.5232e-01&3.8276e-01&	1.6208e-01&	4.9866e-02&	1.1771e-02\\
		
		Order &	0.52905&	1.2397&	1.7005&	2.0827&\\
		
		$10^{-10}$&5.5232e-01	&3.8276e-01	&1.6208e-01	&4.9866e-02	&1.1771e-02 \\ 
	
		Order &	0.52905&	1.2397&	1.7005&	2.0827& \\
		$10^{-11}$&5.5232e-01&3.8276e-01&1.6208e-01&4.9866e-02&	1.1771e-02 \\ 
		
		Order &	0.52905&	1.2397&	1.7005&	2.0828& \\
		
		$10^{-12}$&5.5232e-01&3.8276e-01&1.6208e-01&4.9866e-02&	1.1771e-02 \\ 
		
		Order &	0.52905&	1.2397&	1.7005&	2.0828& \\
		$10^{-13}$&5.5232e-01&3.8276e-01&1.6208e-01&4.9866e-02&	1.1771e-02 \\ 
		
		Order &	0.52905&	1.2397&	1.7005&	2.0828& \\
		$10^{-14}$&5.5232e-01&3.8276e-01&1.6208e-01&4.9866e-02&	1.1771e-02 \\ 
		
		Order &	0.52905&	1.2397&	1.7005&	2.0828& \\
		$10^{-15}$&5.5232e-01&3.8276e-01&1.6208e-01&4.9866e-02&	1.1771e-02 \\ 
		
		Order &	0.52905&	1.2397&	1.7005&	2.0828& \\
		$10^{-16}$&5.5232e-01&3.8276e-01&1.6208e-01&4.9866e-02&	1.1771e-02 \\ 
		
		Order &	0.52905&	1.2397&	1.7005&	2.0828& \\
		$10^{-17}$&5.5232e-01&3.8276e-01&1.6208e-01&4.9866e-02&	1.1771e-02 \\ 
		
		Order &	0.52905&	1.2397&	1.7005&	2.0828& \\
	
\hline
\end{tabular}
\end{minipage}
\end{center}
\end{table}
\begin{table}[h]
\begin{center}
\begin{minipage}{174pt}
\caption{Maximum point-wise error $E^{N}$ and approximate orders of convergence $R^{N}$ for Example \ref{ex-b} when $\mu=10^{-4}.$}\label{eg22}%
\begin{tabular}{@{}llllll@{}}
\hline
\multirow{2}{*}{$\epsilon$ }& \multicolumn{5}{c}{Number of mesh points N}\\
\cline{2-6}
& 64  & 128 & 256 &512 & 1024\\
\hline
	$10^{-8}$&5.9397e-01&	4.4115e-01&	2.8197e-01&	1.6259e-01&	8.8048e-02\\ 
		
		Order  &	0.42911&	0.64574&	0.79429&	0.88487&\\
		$10^{-9}$&7.6741e-01&	5.0315e-01&	2.9775e-01&	1.6443e-01&	8.7022e-02\\ 
		
		Order  &0.60902&	0.75685&	0.85662&	0.91804&\\
		
		$10^{-10}$&8.0508e-01&	5.1239e-01&	2.9860e-01&	1.6361e-01&8.6241e-02\\
		
		Order &0.65187&	0.77903&	0.86797&	0.92381&\\
		$10^{-11}$&8.0927e-01&	5.1325e-01&	2.9859e-01&	1.6346e-01&	8.6127e-02\\
		
		Order &0.65694&	0.78150&	0.86920&	0.92443&\\
		
		$10^{-12}$&8.0970e-01&	5.1334e-01&	2.9858e-01&	1.6344e-01&	8.6114e-02 \\ 
		
		Order&	0.65746&	0.78176&	0.86932&	0.92450& \\
		$10^{-13}$&8.0974e-01&	5.1335e-01&	2.9858e-01&	1.6344e-01&	8.6106e-02 \\ 
		
		Order&0.65750&	0.78179&	0.86937&	0.92459& \\
		
		$10^{-14}$&8.0974e-01&	5.1336e-01&	2.9858e-01&	1.6345e-01&	8.6022e-02 \\ 
		
		Order &	0.65749&	0.78182&	0.86929&	0.92607& \\
		$10^{-15}$&8.0976e-01&	5.1344e-01&	2.9857e-01&	1.6328e-01&	8.6538e-02\\ 
		
		Order &0.65729&	0.78210&	0.87070&	0.91597& \\
		$10^{-16}$&8.0948e-01&	5.1325e-01&	2.9734e-01&	1.5930e-01&	9.4156e-02 \\ 
		
		Order &0.65734&	0.78751&	0.90032&	0.75868& \\
		$10^{-17}$&8.0711e-01&	5.1523e-01&	2.8102e-01&	1.4310e-01&	4.9626e-02 \\ 
		
		Order &0.64754&	0.87454&	0.97363&	1.5278& \\
\hline
\end{tabular}
\end{minipage}
\end{center}
\end{table}
\begin{table}[h]
\begin{center}
\begin{minipage}{174pt}
\caption{Comparison of order of convergence using Shishkin mesh and Shishkin-Bakvalov mesh of Example \ref{ex-a} for $\epsilon=10^{-8}.$}\label{comp}%
\begin{tabular}{@{}llllll@{}}
\hline
\multirow{2}{*}{$\mu$ }&\multirow{2}{*}{Mesh }& \multicolumn{4}{c}{Number of mesh points N}\\
\cline{3-6}
& & 64  & 128 & 256 &512 \\
\hline
	\multirow{2}{*}{$10^{-5}$}&S-mesh&0.23087&	0.40876&	0.57128&	0.68814\\

		&S-B mesh&0.64591&	0.79977&	0.88581&	0.93482	\\
	
		\multirow{2}{*}{$10^{-6}$}&S-mesh&0.27379&	0.46997&	0.63313&	0.73471\\
		&S-B mesh&0.73267&	0.86950&	0.92879&	0.95837	\\
		
		\multirow{2}{*}{$10^{-7}$}&S-mesh&0.27851&	0.47689&	0.64024&	0.74010\\
	
		&S-B mesh&0.74265&	0.87771&	0.93392&	0.96120	\\
	
		\multirow{2}{*}{$10^{-8}$}&S-mesh&	0.27899&	0.47759&	0.64096&	0.74064\\
	
		&S-B mesh&	0.74366&0.87854&	0.93444	&0.96149\\
	
		\multirow{2}{*}{$10^{-9}$}&S-mesh&0.27904&	0.47766&	0.64103&	0.74070\\
	
		&S-B mesh&0.74376&	0.87863&	0.93450&	0.96152\\
	
		\multirow{2}{*}{$10^{-10}$}&S-mesh&	0.27904&	0.47767&	0.64104&	0.74070\\
	
		&S-B mesh&0.74377&	0.87864&	0.93450&	0.96152\\
	
		\multirow{2}{*}{$10^{-11}$}&S-mesh&	0.27904&	0.47767&	0.64104&	0.74070\\
		
		&S-B mesh&	0.74377&	0.87864&	0.93450	&0.96152\\
		
		\multirow{2}{*}{$10^{-12}$}&S-mesh&	0.27904&	0.47767&	0.64104&	0.74070\\
		
		&S-B mesh&	0.74377&	0.87864&	0.93450	&0.96152\\
	
		\multirow{2}{*}{$10^{-13}$}&S-mesh&	0.27904&	0.47767&	0.64104&	0.74070\\
		
		&S-B mesh&	0.74377&	0.87864&	0.93450	&0.96152\\
		
		\multirow{2}{*}{$10^{-14}$}&S-mesh&	0.27904&	0.47767&	0.64104&	0.74070\\
		
		&S-B mesh&	0.74377&	0.87864&	0.93450	&0.96152\\
\hline
\end{tabular}
\end{minipage}
\end{center}
\end{table}

%\section{Figures}\label{sec8}
\begin{figure}[h]%
\centering
\includegraphics[width=0.9\textwidth]{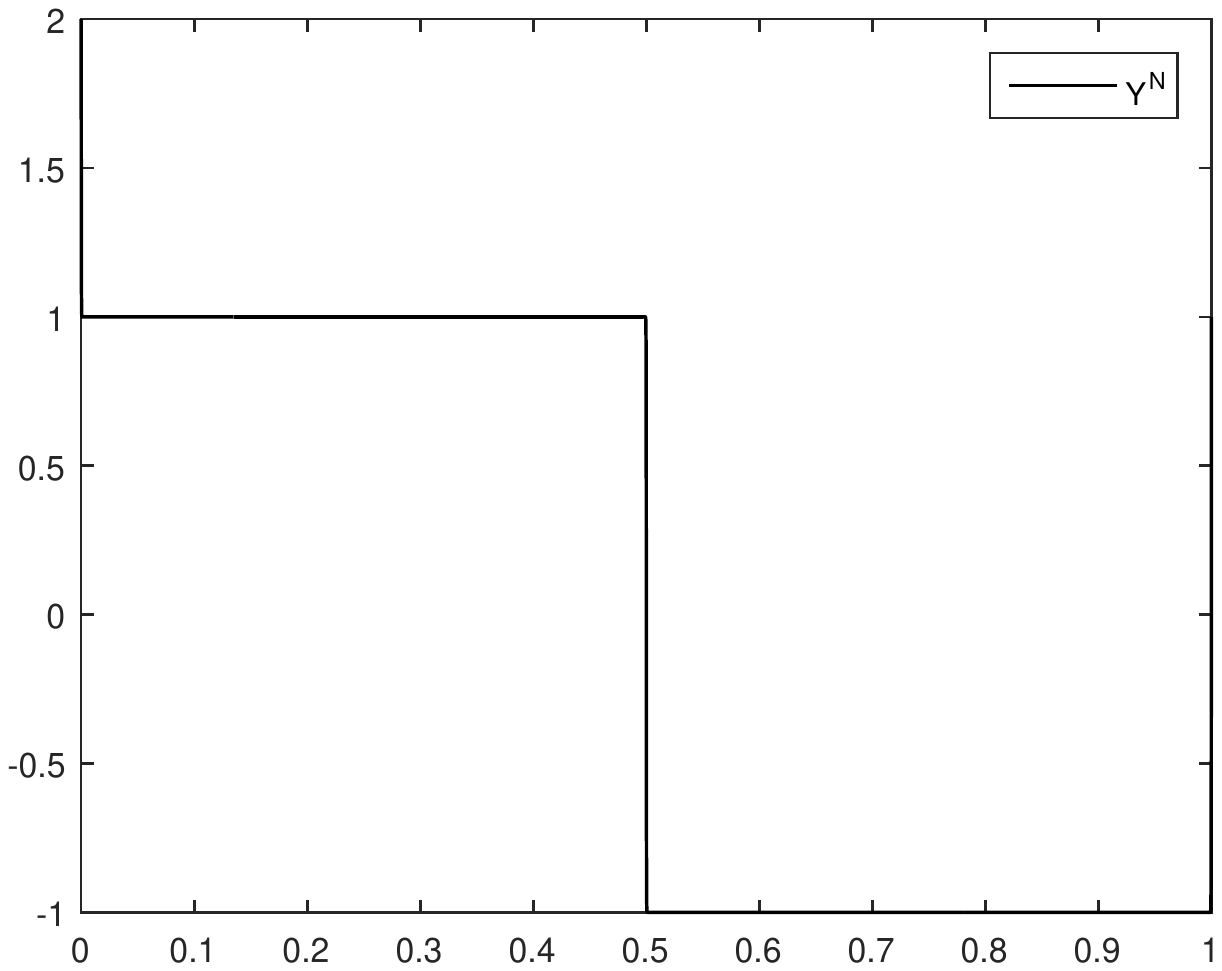}
\caption{Plot of numerical solution for $\epsilon=10^{-8}, \mu=10^{-6}$ when $N=256$ for Example \ref{ex-a}.}\label{fig11}
\end{figure}
\begin{figure}[h]%
\centering
\includegraphics[width=0.9\textwidth]{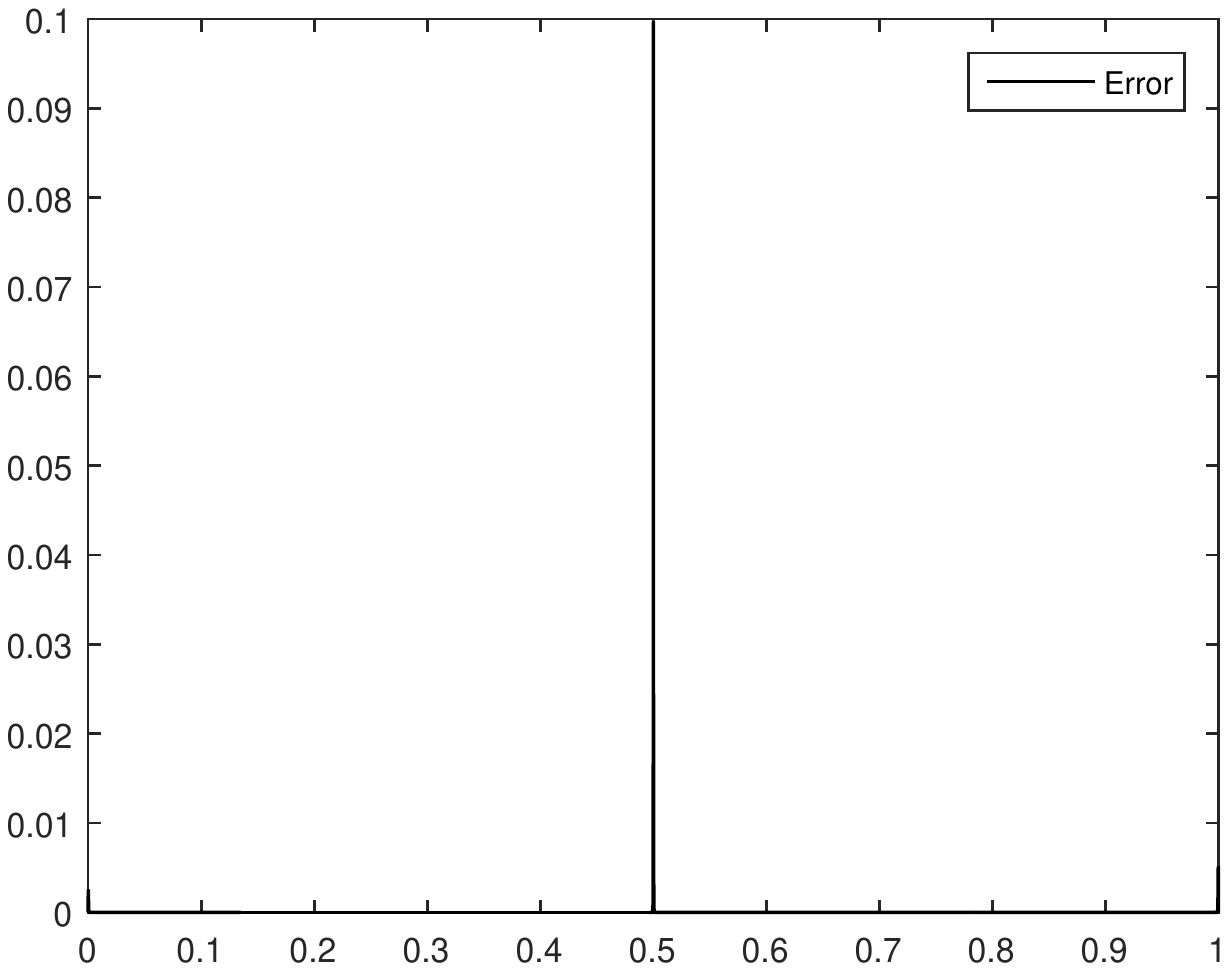}
\caption{Plot of errors for $\epsilon=10^{-8}, \mu=10^{-6}$ when $N=256$ for Example \ref{ex-a}.}\label{fig12}
\end{figure}
\begin{figure}[h]%
\centering
\includegraphics[width=0.9\textwidth]{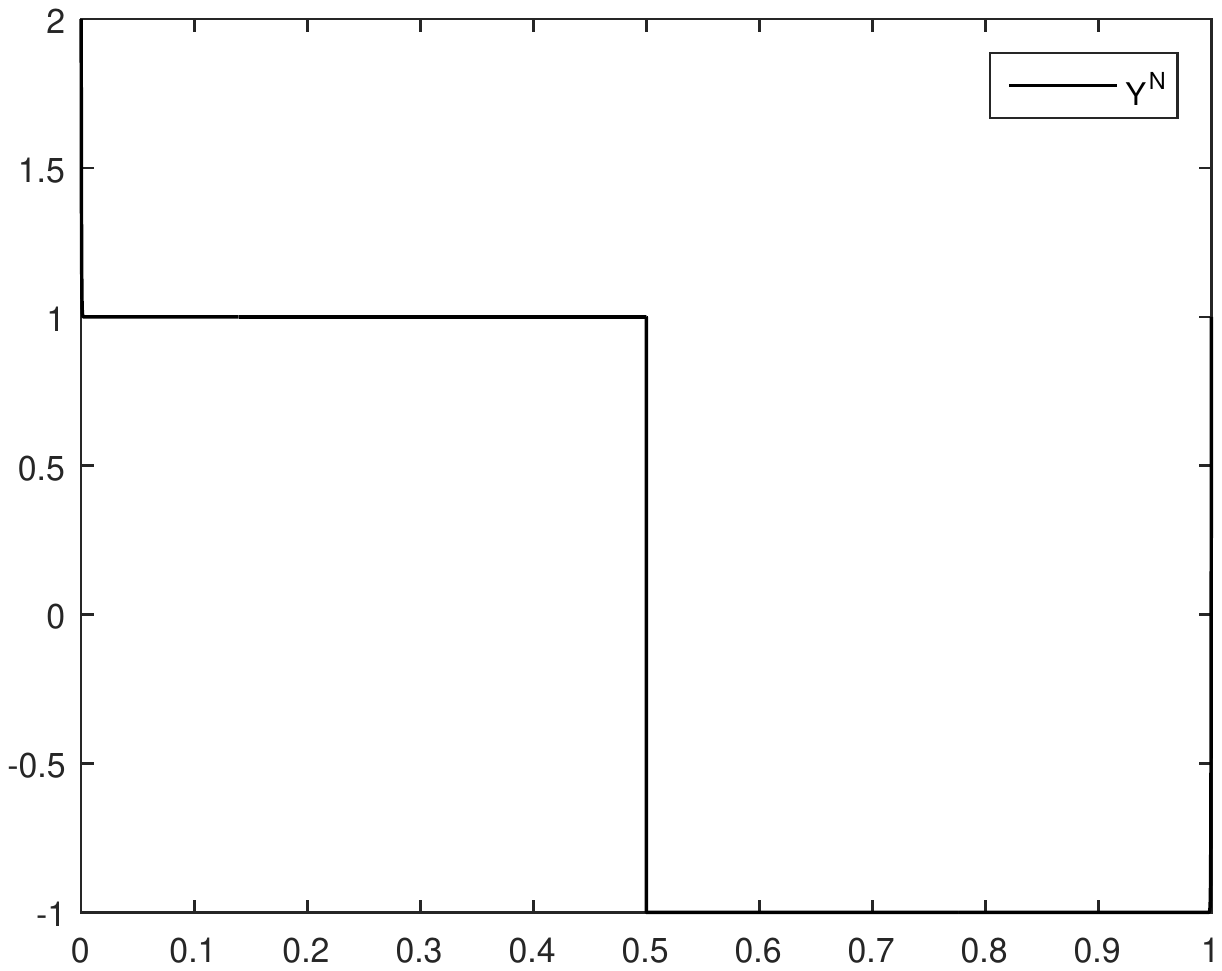}
\caption{Plot of numerical solution for $\epsilon=10^{-12}, \mu=10^{-4}$ when $N=256$ for Example \ref{ex-a}.}\label{fig13}
\end{figure}
\begin{figure}[h]%
\centering
\includegraphics[width=0.9\textwidth]{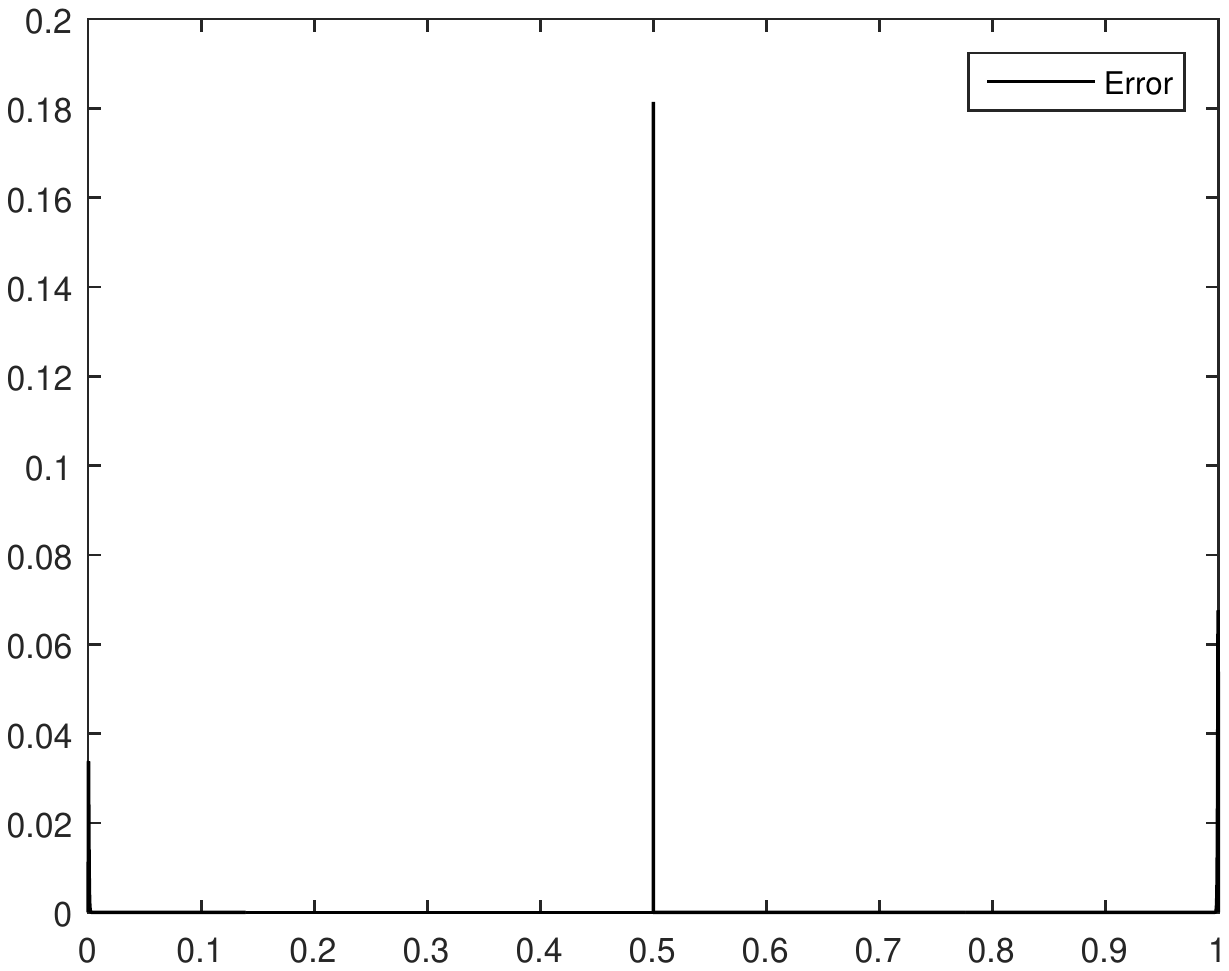}
\caption{Plot of errors for $\epsilon=10^{-12}, \mu=10^{-4}$ when $N=256$ for Example \ref{ex-a}.}\label{fig14}
\end{figure}
\begin{figure}[h]%
\centering
\includegraphics[width=0.9\textwidth]{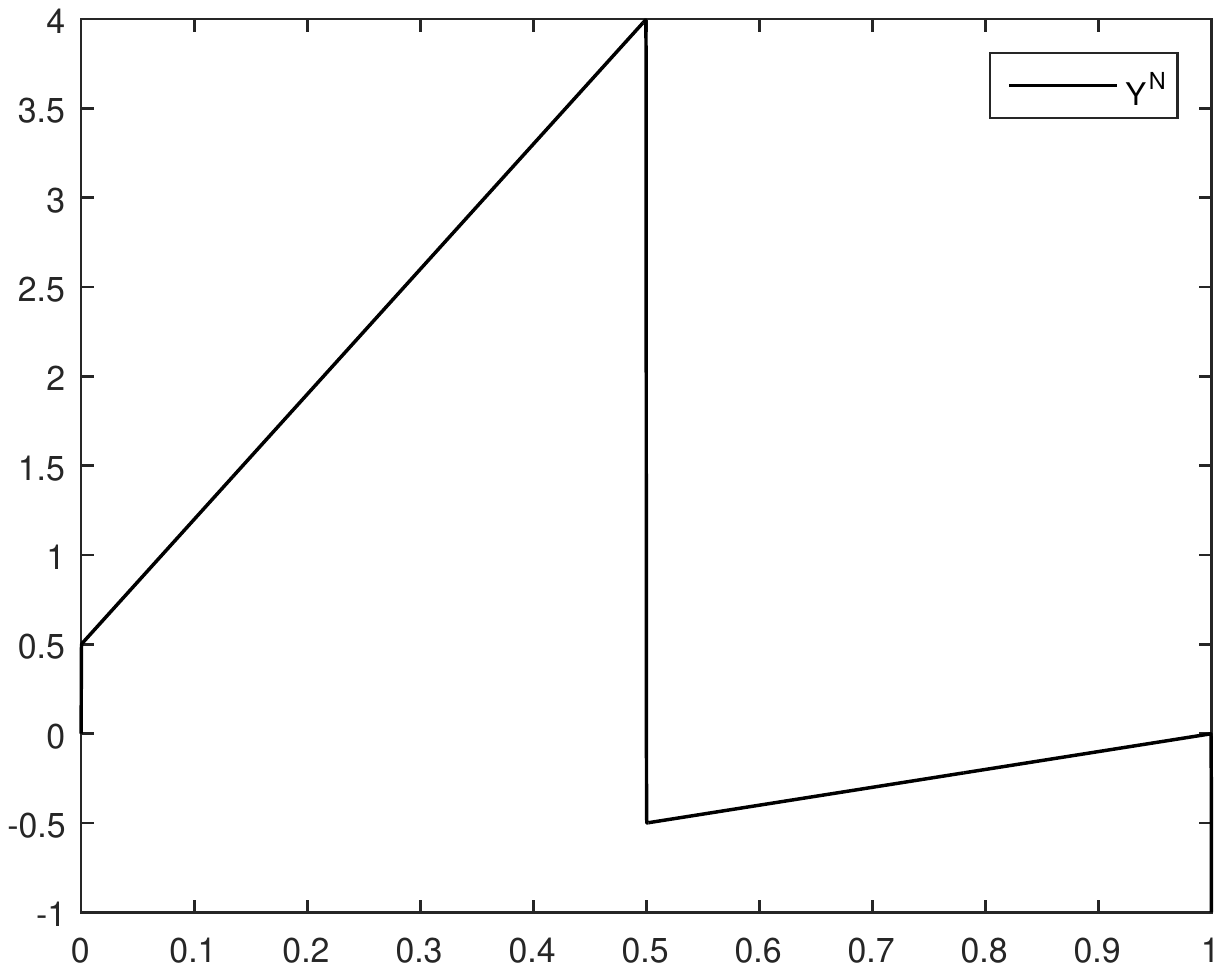}
\caption{Plot of numerical solution for $\epsilon=10^{-8}, \mu=10^{-6}$ when $N=256$ for Example \ref{ex-b}.}\label{fig21}
\end{figure}
\begin{figure}[h]%
\centering
\includegraphics[width=0.9\textwidth]{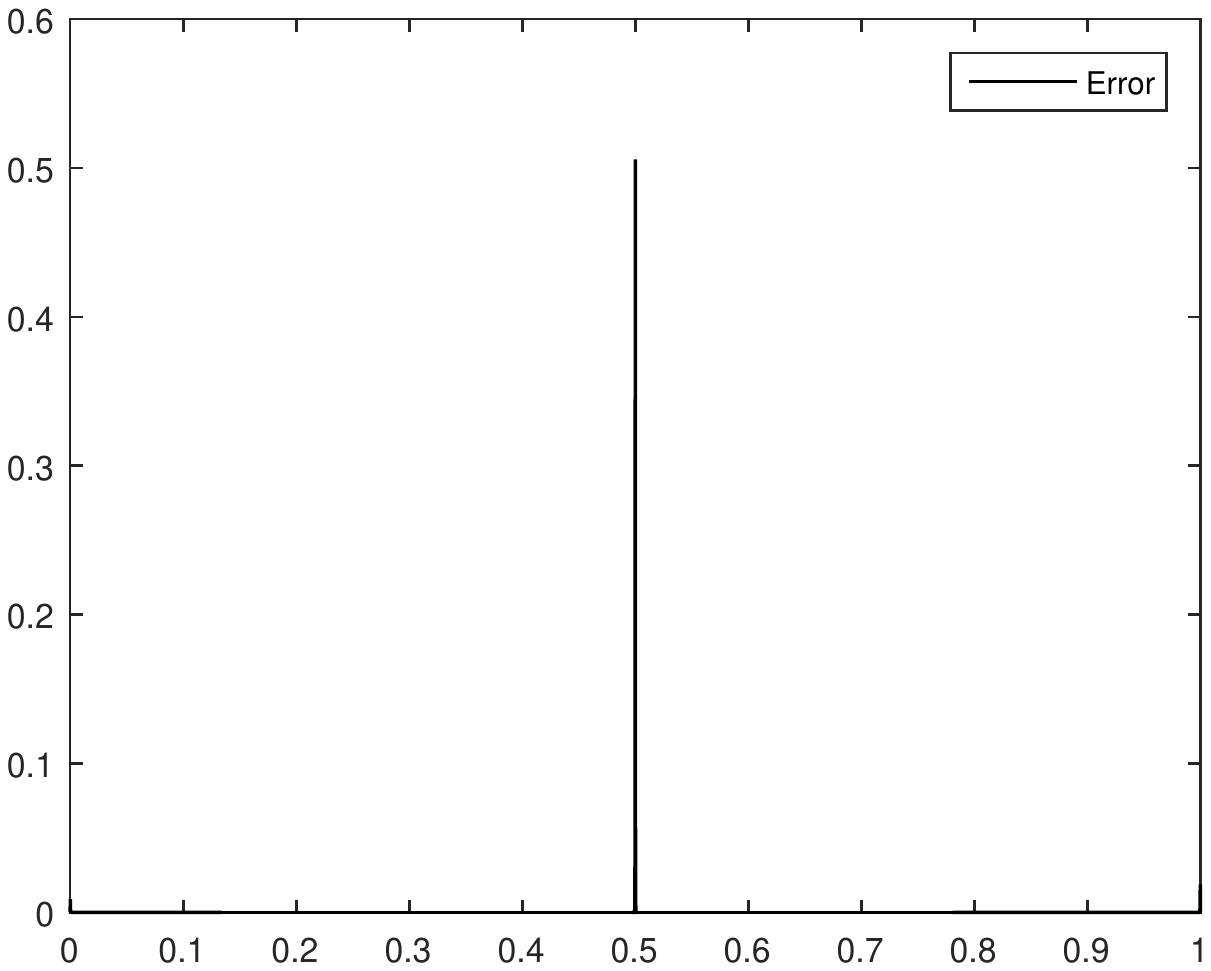}
\caption{Plot of errors for $\epsilon=10^{-8}, \mu=10^{-6}$ when $N=256$ for Example \ref{ex-b}.}\label{fig22}
\end{figure}
\begin{figure}[h]%
\centering
\includegraphics[width=0.9\textwidth]{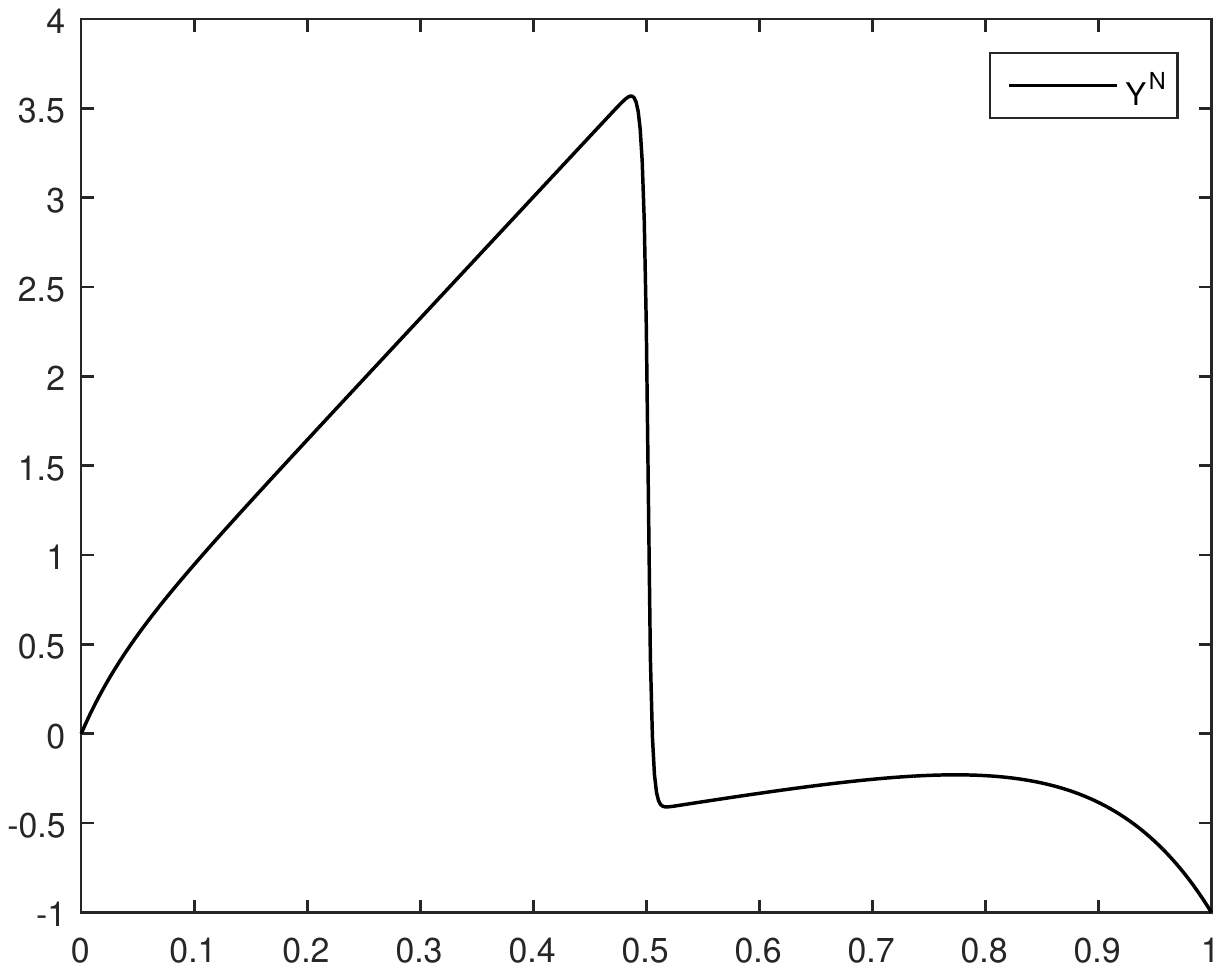}
\caption{Plot of numerical solution for $\epsilon=10^{-12}, \mu=10^{-4}$ when $N=256$ for Example \ref{ex-b}.}\label{fig23}
\end{figure}

\begin{figure}[h]%
\centering
\includegraphics[width=0.9\textwidth]{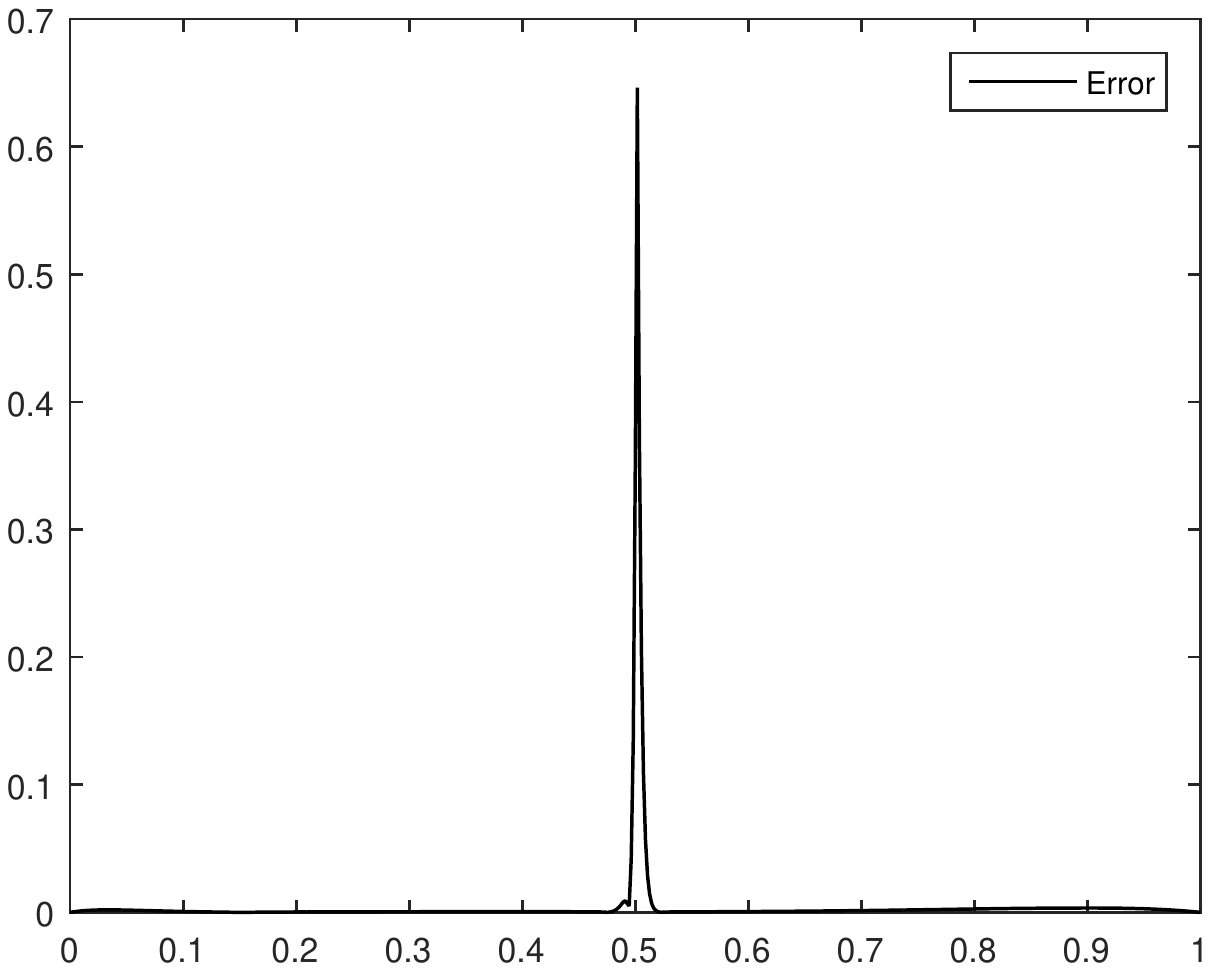}
\caption{Plot of errors for $\epsilon=10^{-12}, \mu=10^{-4}$ when $N=256$ for Example \ref{ex-b}.}\label{fig24}
\end{figure}

\section{Conclusion}\label{sec9}
In this article, we have proposed a Shishkin-Bakhvalov mesh on an upwind scheme to solve the two-parameter singularly perturbed BVP with a discontinuous source term and convection coefficient. At the point of discontinuity, we consider a three-point difference scheme. The theoretical error estimates prove that the proposed scheme is first-order convergent in the maximum norm. The use of the Shishkin-Bakhvalov mesh helps in achieving the first-order convergence. The numerical results presented confirm the theoretical error estimates obtained. The numerical order of convergence approaches one as the number of mesh points increases. A comparison table between the numerical order of convergence obtained through the Shishkin mesh and the Shishkin-Bakhvalov mesh shows the efficiency of the mesh used.


\begin{thebibliography}{9}

\bibitem{Alh} Alhumaizi, K.: Flux limiting solution techniques for simulation of reaction-diffusion-convection system. Commun. Nonlinear Sci. Numeri. simul. 12(6), 953-965 (2007)
	
\bibitem{Cen} Cen, Z.: A hybrid difference scheme for a singularly perturbed convection-diffusion problem with discontinuous convection coefficient.  Appl. Math. Comput. 169(1), 689-699 (2005)
	
\bibitem{CPS}  Chandru, M., Prabha, T., Shanthi, V.: A parameter robust higher order numerical method for singularly perturbed two parameter problems with non-smooth data. J. Comput. Appl. Math. 309, 11–27 (2017)  
	
\bibitem{DM} Doolan, E.P., Miller, J.J.H., Schilders, W.H.: Uniform Numerical Methods for Problems with Initial and Boundary Layers. , Boole Press, Vol 1 (1980)

\bibitem{FHMRS1}	Farrell, P.A., Hegarty, A.F.,  Miller, J.J.H., O'Riordan,  E., Shishkin,  G.I.: Robust computational techniques for boundary layers. Chapman and Hall/CRC, Boca Raton, FL, (2000)

\bibitem{FMRS} 	Farrell, P.A., Miller,  J.J.H., O'Riordan, E., Shishkin, G.I.: Singularly perturbed differential equations with discontinuous source terms. in: J.J.H. miller, G.I. Shishkin, L. Vulkov (Eds.), Analytical and Numerical Method for Convection-Dominated and Singularly Perturbed Problems, Nova Science Publishers, Inc., New York, 23-32 (1998)
	
\bibitem{FH} Farrell, P.A., Hegarty,  A.F.,  Miller,  J.J.H.,  O'Riordan, E.,  Shishkin, G.I., Global maximum norm parameter-uniform numerical method for a singularly perturbed convection-diffusion problem with discontinuous convection coefficient. Math. Comput. Modelling 40 (11–12), 1375-1392 (2004) 
	
\bibitem{FHMRS} Farrell,  P.A.,  Hegarty, A.F.,  Miller,J .J.H.,  O'Riordan, E.,  Shishkin,G.I.: Singularly perturbed convection-diffusion problems with boundary and weak interior layers. J. Comput. Appl. Math. 166, 133-151 (2004)	

\bibitem{GRP}  Gracia, J. L., O'Riordan, E.; Pickett, M. L. A parameter robust second order numerical method for a singularly perturbed two-parameter problem. Appl. Numer. Math., 56(7), 962-980 (2006)
	
\bibitem{Hir} Hirsch, C.: Numericcal computation of internal and external flows. Wiley, Chichester, Vol I, (1990)
	
\bibitem{PHS} Polak, S., Heijer, C.D., Schilders, W.: Semiconductor device modelling from the numerical point of view. Int. J. Numer. Methods Eng. 24, 763-838 (1987)
	
\bibitem{KL} Kreiss, H. O., Lorenz, J.: Initial- boundary value problems and the Navier-Stokes equations.  Classics in Appl. Math. SIAM,  Philadelphia, PA (Reprint of 1989 edition), Vol 47, (2004)

\bibitem{Linss99} T. Lin{\ss}, An upwind difference scheme on a novel Shishkin-type mesh for a linear convection–diffusion problem.
J. of Comput and Appl. Math.,110(1), 93-104(1999)
	
\bibitem{Lins02} Lin{\ss}, T.: Finite difference schemes for convection-diffusion problems with a concentrated source and a discontnuous convection field.  Comput. Methods Appl. Math. 2(1), 41-49 (2002) 

\bibitem{OM1}  O'Malley Jr, R.E.: Two-parameter singular perturbation problems for second order equations.  J. Math. Mech., 16, 1143-1164 (1967).

\bibitem{OM} O'Malley Jr, R.E.: Introduction to Singular Perturbations. Academic Press, New York, (1974).

\bibitem{OM2}O'Malley Jr, R.E.: Singular Perturbation Methods for Ordinary Differential Equations. Springer, New York, (1990).

\bibitem{RPS} O'Riordan, E.; Pickett, M. L.; Shishkin, G. I.: Singularly perturbed problems modeling reaction-convection-diffusion processes.  Comput. Methods Appl. Math. 3(3), 424-442(2003)
	
\bibitem{TP}  Prabha, T., Chandru, M., Shanthi, V., Ramos, H.: Discrete approximation for a two-parameter singularly perturbed boundary value problem having discontinuity in convection coefficient and source term. J. of Comput. and Appl. Math. 359,  102-118 (2019)
	
\bibitem{PCS} Prabha, T., Chandru, M., Shanthi, V.: Hybrid Difference Scheme for Singularly Perturbed Reaction-Convection-Diffusion Problem with Boundary and Interior Layers.  Appl. Math. Comput. 31,  237-256 (2017) 

\bibitem{RU}  Roos, H.-G., Uzelac, Z.: The SDFEM for a convection diffusion problem with two small parameters.
Comput. Methods Appl. Math., 3(3), 443-458(2003)
	
\bibitem{REILW} Rap, A., Elliott, L., Ingham, D.B., Lesnic, D., Wen, X.: The inverse source problem for the variable coefficients confection-diffusion equation. Inverse Probl. sci. Eng. 15, 413-440 (2007)

\bibitem{SRN06}	Shanthi, V., Ramanujam, N., Natesan, S.: Fitted mesh method for singularly perturbed reaction-convection-diffusion problems with boundary and interior layers. J. Appl. Math. Comput. 22( 1-2), 49-65 (2006) 

\bibitem{Z} Zahra,  W.K., El Mhlawy, A.M.: Numerical solution of two-parameter singularly perturbed boundary value problems via exponential spline. J. King Saud Univ., 25(3), 201-208(2013)
	
\bibitem{ZW} Zahra, W.K., Daele, M.V.: \textit{Discrete spline solution of singularly perturbed problem with two small parameters on a Shishkin-Type mesh}, Comput. Math. and Model., 29(5), 1-15(2018) 

\end{thebibliography}
\end{document}